\documentclass[11pt,letterpaper,reqno]{amsart}
\usepackage[foot]{amsaddr}
\usepackage[normalem]{ulem}
\usepackage{amsmath, amsthm, amssymb,bbm, tikz, pgfplots} 
\usepackage[utf8]{inputenc} 
\usepackage{graphicx} 
\usepackage[small]{caption}
\usetikzlibrary{arrows,calc}
\numberwithin{equation}{section}

\usepackage{lmodern} 
\usepackage[T1]{fontenc} 

\usepackage{enumerate} 
\usepackage{verbatim} 
\usepackage{color}
\usepackage{tabularx}

\usepackage[colorlinks,linkcolor=blue, citecolor=blue, urlcolor=blue]{hyperref} 

\usepackage[letterpaper,nomarginpar]{geometry}

\usepackage[capitalise]{cleveref}

\newcommand{\R}{\mathbb{R}}
\newcommand{\E}{\mathbb{E}}

\newcommand{\N}{\mathbb{N}}
\newcommand{\Z}{\mathbb{Z}}

\renewcommand{\S}{\mathbb{S}}

\newcommand{\cC}{\mathcal{C}}

\newcommand{\cA}{\mathcal{A}}
\newcommand{\cB}{\mathcal B}

\newcommand{\fC}{\mathfrak{C}}
\newcommand{\cE}{\mathcal{E}}

\newcommand{\cF}{\mathcal F}
\newcommand{\cM}{\mathcal{M}}
\newcommand{\cN}{\mathcal{N}}
\newcommand{\cD}{\mathcal{D}}
\newcommand{\cS}{\mathcal{S}}

\renewcommand{\d}{\mathrm{d}}
\newcommand{\Rd}{\R^{d}}
\renewcommand{\P}{\mathbb{P}}

\newcommand{\ind}[1]{\mathbbm{1}_{\{ #1\}}}
\newcommand{\indset}[1]{\mathbbm{1}_{#1}}
\newcommand{\egaldistr}{\overset{(d)}{=}}

\newtheorem{maintheorem}{Theorem}
\newtheorem{maincoro}[maintheorem]{Corollary}

\newtheorem{theorem}{Theorem}[section]
\newtheorem*{theorem*}{Theorem}
\newtheorem{lemma}[theorem]{Lemma}
\newtheorem{claim}[theorem]{Claim}
\newtheorem{proposition}[theorem]{Proposition}

\newtheorem{corollary}[theorem]{Corollary}

\theoremstyle{definition}{

\newtheorem{definition}[theorem]{Definition}
\newtheorem*{definition*}{Definition}

\newtheorem*{question*}{Question}
\newtheorem*{example*}{Example}
\newtheorem*{examples*}{Examples}
\newtheorem{remark}[theorem]{Remark}
\newtheorem*{remark*}{Remark}

}

\newcommand{\norm}[1]{\left\Vert #1 \right\Vert}

\renewcommand{\bar}[1]{\overline{#1}}

\newcommand{\given}{\;\big|\;}
\newcommand\abs[1]{\left|#1\right|}

\newcommand{\ep}{\epsilon}

\newcommand{\cR}{\mathcal{R}}

\newcommand{\Y}{\mathbf{Y}}
\newcommand{\B}{\mathbf{B}}
\newcommand{\bA}{\hat{A}}

\newcommand{\bZ}{\tilde{Z}}

\newcommand{\bB}{\bar{\mathbf{B}}}
\newcommand{\Pb}{\P^{(b)}}

\newcommand{\exit}{\mathrm{exit}}

\newcommand{\fc}{\mathfrak{c}}

\newcommand{\Ebump}{E_{\xi}^{\Bumpeq}}

\newcommand{\Id}{\mathrm{Id}}

\newcommand{\h}{\mathfrak{h}_{\mathcal{C}}^{\epsilon}}
\newcommand{\hcirc}{\mathfrak{h}_{\mathcal{C}}^{\epsilon \circ}}
\newcommand{\hs}{\bar{h}^{\epsilon}}
\newcommand{\hsb}{\bar{h}^{\epsilon \circ}}
\renewcommand{\land}{\mathrm{land}}
\newcommand{\cand}{\mathrm{cand}}
\newcommand{\good}{\mathrm{good}}

\makeatletter
\crefname{claim}{Claim}{Claims}
\makeatother

\title{The shape of the front of multidimensional branching Brownian motion}

\author[Y. H. Kim]{Yujin H. Kim}
\address{Y.\ H.\ Kim\hfill\break
Courant Institute\\ New York University\\
251 Mercer Street\\ New York, NY 10012, USA.}
\email{yujin.kim@courant.nyu.edu}

\author[O. Zeitouni]{Ofer Zeitouni}
\address{O.\ Zeitouni\hfill\break
Department of Mathematics\\
Weizmann Institute of Science\\
POB 26,
Rehovot 76100, Israel\\
and
Courant Institute\\
New York University.}
\email{ofer.zeitouni@weizmann.ac.il}

\begin{document}

\begin{abstract}
We study the shape of the outer envelope of a branching Brownian motion (BBM) in $\mathbb{R}^d$, $d\geq 2$. We focus on the extremal particles: those whose norm is within $O(1)$ of the maximal norm amongst the particles alive at time $t$. Our main result is a scaling limit, with exponent $3/2$, for the outer-envelope of the BBM around each extremal particle (the \emph{front}); the scaling limit is a continuous random surface given explicitly in terms of a Bessel(3) process. Towards this end, we introduce a point process that captures the full landscape around each extremal particle and show convergence in distribution to an explicit point process. This complements the global description of the extremal process given in Berestycki et.\ al.\ (Ann.\ Probab.\ \textbf{52} (2024), no.\ 3, 955--982), where the local behavior at directions transversal to the radial component of the extremal particles is not addressed.
\end{abstract}


{\mbox{}
\maketitle
}
\vspace{-.5cm}

\section{Introduction}
Fix the dimension $d\geq 2$. 
We  study standard, binary branching Brownian motion (henceforth, ``BBM'') in $\R^d$. This is a stochastic branching particle system, defined as follows. Begin with a single particle at the origin $0 \in \R^d$. 
This particle moves as a standard Brownian 
motion in $\R^d$ and splits into two particles after an amount of time distributed as an independent exponential random variable 
of rate $1$. Each of these particles then continues independently via the same process.
Let $\cN_t$ denote the set of particles alive at time $t$, and for each $u \in \cN_t$, let $\B_t(u)  \in \R^d$ denote the position of $u$ at time $t$. For $u \in \cN_t$ and $s \in [0, t]$, let $\B_s(u) \in \R^d$ denote the position of the unique ancestor of $u$ alive at time $s$.

This paper addresses the following question. Observe a BBM in dimension $d$ at some large time (see Figure~\ref{fig:extremal-cluster}, top-left panel, 
 for a fixed-time $d=2$ simulation). Draw the outer envelope of the process in some way--- this is a
$(d-1)$-dimensional random surface. What can we say about it?

In \cite{Biggins}, Biggins showed that the convex hull of the BBM at time $t$, normalized by $\sqrt2t$, converges to the unit sphere.
Moreover, it is not hard to see that for large times $t$, the BBM actually
fills
a $d$-dimensional ball of some growing radius $f(t)$ in the following sense: if we replace each particle with a 
ball of radius, say, $1$, then our particle system will fill up a ball centered at the origin of some radius $f(t)$.  As explained in \cite[Section~1]{Mallein15}, 
G\"artner's theorem
in \cite{gartner} implies that 
$f(t) \approx \sqrt2 t - \frac{d+2}{2\sqrt2} \log t$.

On the other hand, there will be exceptional particles that travel much farther from the origin, creating long, thin spikes away from the aforementioned
ball. In \cite{Mallein15}, Mallein showed that the largest norm amongst particles in $\cN_t$ is tight around $m_t(d) := \sqrt2t + \frac{d-4}{2\sqrt2} \log t$; convergence of the largest norm after recentering by $m_t(d)$ was shown in \cite{KLZ21}. It is therefore natural to ask about the shape of
the frontier of the process formed by the \emph{extremal particles}: those particles at time $t$ whose norms are within $O(1)$ of the largest norm, or equivalently, of $m_t(d)$.
For instance, consider the particle $u^* \in \cN_t$ that travelled the furthest from the origin, as well as the cluster of particles 
of $O(1)$ distance from it.
Draw the outer-envelope of this cluster in some sensible way (eg., Figure~\ref{fig:extremal-cluster}, top-right panel).
This shape is roughly what we refer to as the front behind $u^*$.
Our main result, \cref{thm:extremal-front}, is a scaling limit of exponent $3/2$ for the front around \emph{any} extremal particle to an explicit random surface. This is stated precisely in the next subsection.

\begin{figure}
    \includegraphics[width = 0.38\textwidth]{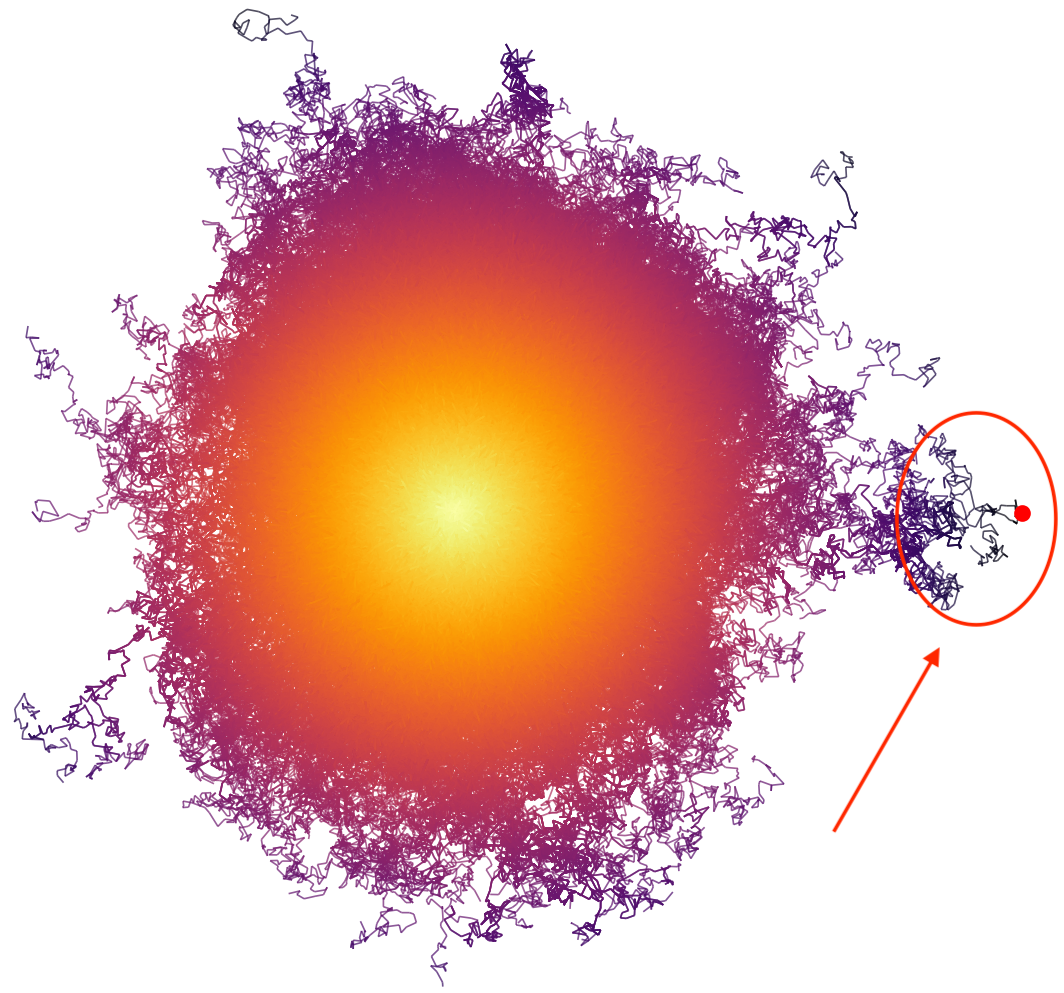}
    \hspace{5 em}
    \includegraphics[height = 0.25\textheight]{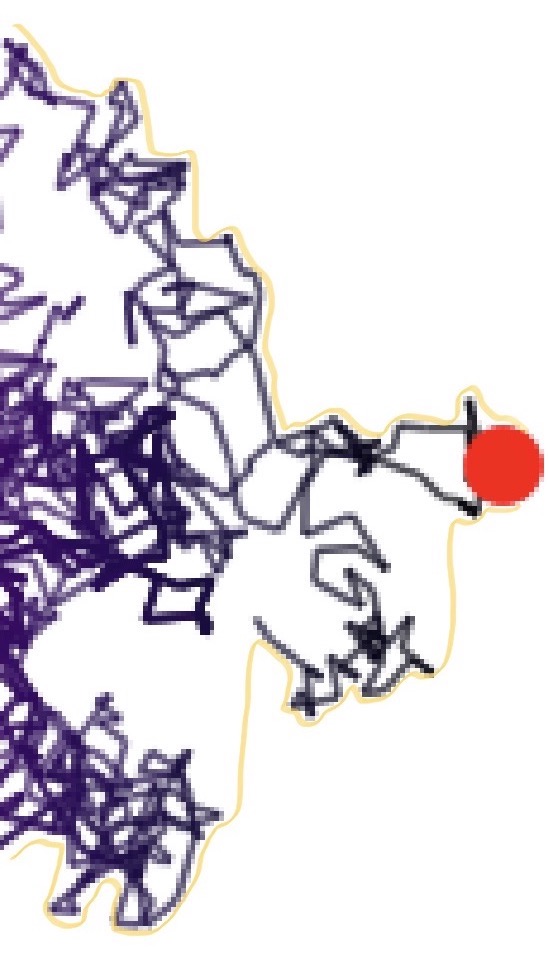}
\includegraphics[width=0.5\textwidth]{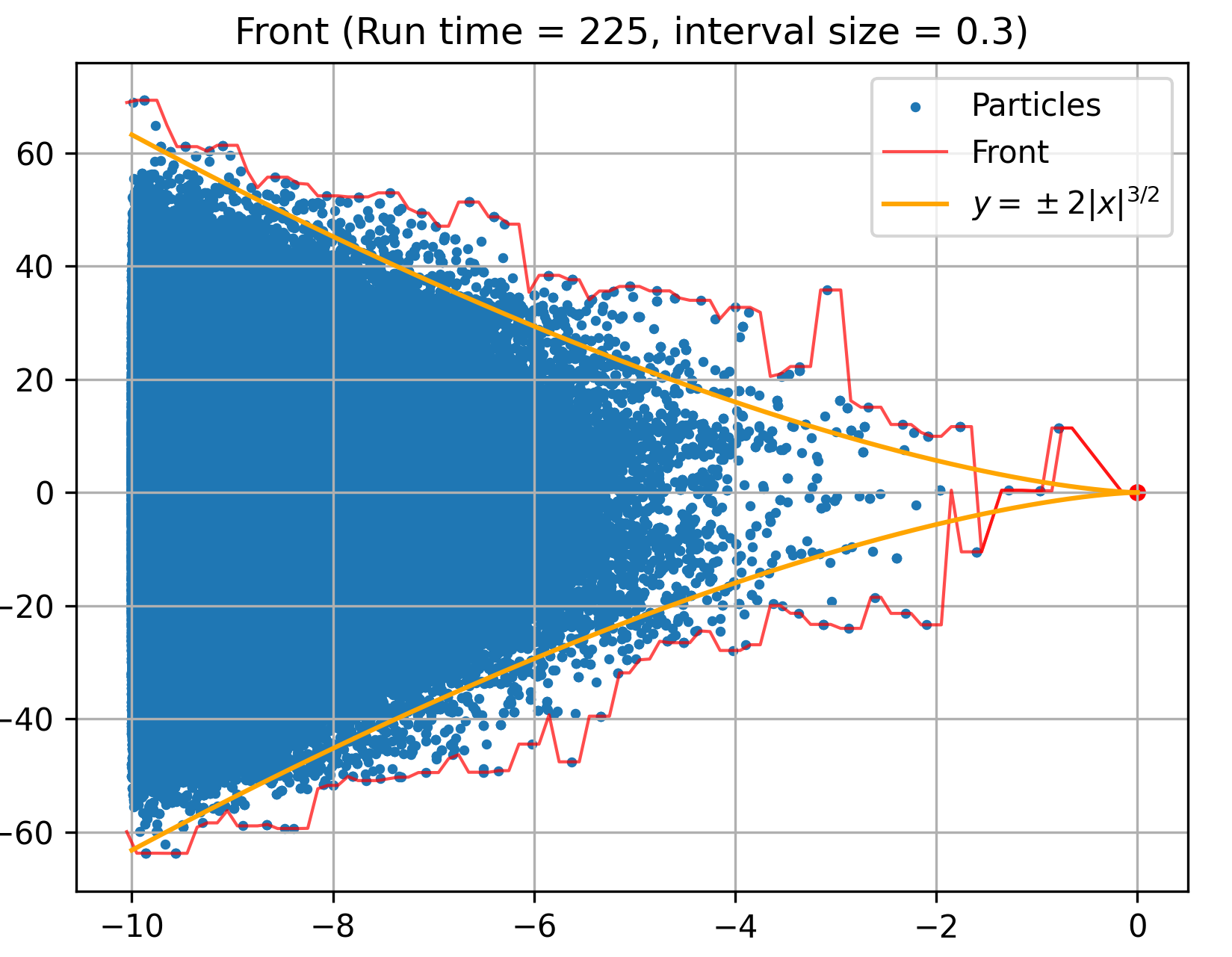}
\caption{(Top-left) A simulation of 2D BBM run until time $t= 8.5$, with trajectories of particles colored according to norm, $u^*$ marked in red, and  
particles closest to $u^*$
circled in red. 
(Top-right) Roughly, the front around $u^*$ in the top-left simulation.
(Bottom) A  simulation of the front run until $t=225$. 
In blue, the particles constituting the point process $\cC_t$.
In solid orange, the deterministic curves $y=\pm 2|x|^{3/2}$.
In red, the front $\mathfrak{h}_{\cC_t}^{\ep}$ 
defined with intervals of size $0.3$ (i.e., $p_i^{(1)} \in (-s, -s+0.3]$ in \cref{def:front}) for a better rendition.}
\label{fig:extremal-cluster}
\end{figure}

\subsection{The front of multidimensional BBM}
We now  move to define the front of the  BBM (\cref{def:front-bbm-extremal-cluster}).  For ease of exposition, we focus on the front around the maximal particle. We explain in \cref{rk:general-front} that the front around any extremal particle has the same scaling limit.

We begin with some notation. 
The line $\mathrm{span}\{\mathbf{e}_1\} \subset \R^d$ will play a particular role in what follows, and so for $u\in \cN_t$, we write
\[
    \B_t(u) = (X_t(u), \Y_t(u)) \in \R \times \R^{d-1}\,.
\]
For a vector $\mathbf v$ in $\R^n$, $n\in \N$, we  write $\arg (\mathbf v) := \mathbf v / \|\mathbf v\| \in \S^{n-1}$.
We also write $\theta_t(u)=\arg(\B_t(u))$.

\begin{definition}[Front of a point process] \label{def:front}
Fix $d\geq 2$, and
consider any simple point process $\mathcal{P}$ on $\R^d$ that has the origin as its ``right-most'' point; that is, the point process may be expressed as
$\mathcal{P} = \delta_0 + \sum_{i\in \N} \delta_{p_i}$, where 
$p_i := (p_i^{(1)}, p_i^{(2)}) \in \R_- \times \R^{d-1}$. 
For a fixed discretization parameter $\ep\in(0,1)$, we define \textit{the front of $\mathcal{P}$} to be the process
\begin{align*}
    \mathfrak{h}_{\mathcal{P}}^{\ep}(s,\theta) :=  \max \big\{ \|p_i^{(2)}\| : p_i \in \mathcal{P} \,,\, p_i^{(1)} \in (-s, -s+1] \,,\ |\theta\cdot\arg (p_i^{(2)})|\geq 1-\ep  \big\}\,,
\end{align*}
for $(s,\theta) \in [0,\infty) \times \S^{d-2}$,
where $\S^0 := \{\pm 1\}$ and the maximum element of the empty set is defined to be $0$.
\end{definition}
In a slight abuse of terminology, we will refer to both the front as well as its graph ``the front''.
 Note that the front is a $(d-1)$-dimensional random hypersurface, and for $d=2$, the front consists of two curves: $s \mapsto \mathfrak{h}_{\mathcal{P}}^{\ep}(s,+1)$ and $s \mapsto \mathfrak{h}_{\mathcal{P}}^{\ep}(s,-1)$.

\begin{definition}[Front of the BBM, extremal cluster]
\label{def:front-bbm-extremal-cluster}
Fix $d\geq 2$ and $\ep \in (0,1)$.
Let $\mathcal{R}_{\theta}: \R^d \to \R^d$ be the rotation map sending $\theta$ to $\mathbf{e}_1$ and acting as the identity on $\mathrm{span}\{\theta, \mathbf{e}_1\}^{\perp}$. Recall $u^*$ denotes the particle in $\cN_t$ of maximal norm. Define the  \textit{extremal cluster} at time $t$ to be the point process
\[
    \cC_t := \sum_{v\in \cN_t} \delta_{\cR_{\theta_t(u^*)}(\B_t(v) - \B_t(u^*))} \,.
\]
In words, $\cC_t$ describes the cluster of BBM particles around the furthest particle from the origin, rotated so that the $\mathbf{e}_1$ direction corresponds to the argument of this particle. We call $\mathfrak{h}_{\cC_t}^{\ep}$
 the \emph{front of the BBM} at time $t$, or simply, \emph{the front}.
\end{definition}

\begin{figure}
\begin{tikzpicture}[scale=1.6]
\draw[->] (3,0)--(-3.5,0);
\draw[->] (3,-3)--(3,3);
\draw[-] (-1,0)--(-1,3);
\draw[-] (-1,0)--(-1,-3);
\draw[-] (-1.2,0)--(-1.2,3);
\draw[-] (-1.2,0)--(-1.2,-3);
\node[right, thick, xshift=-58pt, yshift=5pt] {$\bullet$};
\node[right, thick, xshift=-57.5pt, yshift=10pt] {$\bullet$};
\node[right, thick, xshift=-56pt, yshift=15pt] {$\bullet$};
\node[right, thick, xshift=-57pt, yshift=25pt] {$\bullet$};
\node[right, thick, xshift=-55pt, yshift=20pt] {$\bullet$};
\node[right, thick, xshift=-58.5pt, yshift=35pt] {$\bullet$};
\node[right, thick, xshift=-55.5pt, yshift=48pt] {$\bullet$};
\node[right, thick, xshift=-57pt, yshift=68pt] {$\bullet$};
\node[right, thick, xshift=-57pt, yshift=-5.5pt] {$\bullet$};
\node[right, thick, xshift=-55.5pt, yshift=-11pt] {$\bullet$};
\node[right, thick, xshift=-57pt, yshift=-14pt] {$\bullet$};
\node[right, thick, xshift=-56pt, yshift=-27pt] {$\bullet$};
\node[right, thick, xshift=-55.5pt, yshift=-22pt] {$\bullet$};
\node[right, thick, xshift=-58pt, yshift=-35pt] {$\bullet$};
\node[right, thick, xshift=-55pt, yshift=-50pt] {$\bullet$};
\node[right, thick, xshift=-57.5pt, yshift=-69pt] {$\bullet$};
\node[right, thick, xshift=130pt, yshift=0pt] {$\bullet$};

\draw[<->] (-1,2)--(3,2);
\node[right, thick, xshift=10pt, yshift=100pt] {$t=1, x=-L+1$};
\draw[<->] (-1.2,2.7)--(3,2.7);
\node[right, thick, xshift=10pt, yshift=130pt] {$t=1, x=-L$};
\draw[-,thick,red] (-1.1,1.5)--(-0.9,1.55)--(-0.7,1.4)--(-0.4,1)--(-0.1,0.8)--
(0.2,0.6)--(0.5,0.55)--(0.8,0.50)--(1.2,0.25)--(1.4,0.2)--(3,0);
\draw[-,thick,blue] (-1.1,-1.5)--(-0.9,-1.55)--(-0.7,-1.4)--(-0.4,-1)--(-0.1,-0.8)--
(0.2,-0.6)--(0.5,-0.55)--(0.8,-0.50)--(1.2,-0.25)--(1.4,-0.2)--(3,0);

\node[right, thick, xshift=-160pt, yshift=10pt] {$t$};
\node[right, thick, xshift=135pt, yshift=140pt] {$y$};

\node[right, thick, xshift=-38pt, yshift=5pt,opacity=0.4] {$\bullet$};
\node[right, thick, xshift=-17.5pt, yshift=10pt,opacity=0.4] {$\bullet$};
\node[right, thick, xshift=-6pt, yshift=15pt,opacity=0.4] {$\bullet$};
\node[right, thick, xshift=7pt, yshift=25pt, opacity=0.4] {$\bullet$};
\node[right, thick, xshift=9pt, yshift=15pt,opacity=0.4] {$\bullet$};
\node[right, thick, xshift=17pt, yshift=10pt,opacity=0.4] {$\bullet$};
\node[right, thick, xshift=37pt, yshift=15pt,opacity=0.4] {$\bullet$};
\node[right, thick, xshift=47pt, yshift=8pt,opacity=0.4] {$\bullet$};
\node[right, thick, xshift=57pt, yshift=5pt,opacity=0.4] {$\bullet$};
\node[right, thick, xshift=77pt, yshift=3pt,opacity=0.4] {$\bullet$};
\node[right, thick, xshift=97pt, yshift=2pt,opacity=0.4] {$\bullet$};
\node[right, thick, xshift=-7pt, yshift=28pt,opacity=0.4] {$\bullet$};
\node[right, thick, xshift=25pt, yshift=20pt,opacity=0.4] {$\bullet$};
\node[right, thick, xshift=18.5pt, yshift=25pt,opacity=0.4] {$\bullet$};
\node[right, thick, xshift=-45.5pt, yshift=48pt,opacity=0.4] {$\bullet$};
\node[right, thick, xshift=-37pt, yshift=28pt,opacity=0.4] {$\bullet$};
\node[right, thick, xshift=-47pt, yshift=68pt,opacity=0.4] {$\bullet$};
\node[right, thick, xshift=-39pt, yshift=58pt,opacity=0.4] {$\bullet$};
\node[right, thick, xshift=-29pt, yshift=32pt,opacity=0.4] {$\bullet$};
\node[right, thick, xshift=-25pt, yshift=48pt,opacity=0.4] {$\bullet$};
\node[right, thick, xshift=-16pt, yshift=39pt,opacity=0.4] {$\bullet$};
\node[right, thick, xshift=19pt, yshift=5pt,opacity=0.4] {$\bullet$};

\node[right, thick, xshift=-38pt, yshift=-7pt,opacity=0.4] {$\bullet$};
\node[right, thick, xshift=-17.5pt, yshift=-8pt,opacity=0.4] {$\bullet$};
\node[right, thick, xshift=-6pt, yshift=-21pt,opacity=0.4] {$\bullet$};
\node[right, thick, xshift=8pt, yshift=-28pt, opacity=0.4] {$\bullet$};
\node[right, thick, xshift=10pt, yshift=-6pt,opacity=0.4] {$\bullet$};
\node[right, thick, xshift=19pt, yshift=-19pt,opacity=0.4] {$\bullet$};
\node[right, thick, xshift=38pt, yshift=-19pt,opacity=0.4] {$\bullet$};
\node[right, thick, xshift=49pt, yshift=-7pt,opacity=0.4] {$\bullet$};
\node[right, thick, xshift=58pt, yshift=-6pt,opacity=0.4] {$\bullet$};
\node[right, thick, xshift=76pt, yshift=4pt,opacity=0.4] {$\bullet$};
\node[right, thick, xshift=90pt, yshift=-1pt,opacity=0.4] {$\bullet$};
\node[right, thick, xshift=-9pt, yshift=-27pt,opacity=0.4] {$\bullet$};
\node[right, thick, xshift=26pt, yshift=-23pt,opacity=0.4] {$\bullet$};
\node[right, thick, xshift=-45pt, yshift=-49pt,opacity=0.4] {$\bullet$};
\node[right, thick, xshift=-39pt, yshift=-27pt,opacity=0.4] {$\bullet$};
\node[right, thick, xshift=-48pt, yshift=-67pt,opacity=0.4] {$\bullet$};
\node[right, thick, xshift=-37pt, yshift=-56pt,opacity=0.4] {$\bullet$};
\node[right, thick, xshift=-28pt, yshift=-33pt,opacity=0.4] {$\bullet$};
\node[right, thick, xshift=-24pt, yshift=-47pt,opacity=0.4] {$\bullet$};
\node[right, thick, xshift=-15pt, yshift=-37pt,opacity=0.4] {$\bullet$};
\node[right, thick, xshift=17pt, yshift=-5pt,opacity=0.4] {$\bullet$};

\draw[thick,dashed,variable=\t,domain=0:6,samples=50]
plot ({-\t+3},{0.2*\t^(3/2)});
\draw[thick,dashed,variable=\t,domain=0:6,samples=50]
plot ({-\t+3},{-0.2*\t^(3/2)});
\end{tikzpicture}
\caption{An artistic rendition of the extremal cluster and the front in $d=2$.
Drawn are only
the particles up to distance  $L$
(in the direction $\mathbf{e}_1$)
behind the
extremal particle (at the origin). The front for $\theta=+1$ is drawn in red,
for $\theta=-1$ in blue. The dashed lines correspond to $y=\pm x^{3/2}$.
The particles in the strip $x\in (-L,-L+1]$ used to define the front
at $t=1$ are shaded darker than other particles.
Note the vertical lines used to define the height of the front at $t=1$.}
\label{fig:front}
\end{figure}
See Figure~\ref{fig:extremal-cluster}, bottom panel, for a simulation of the front and Figure~\ref{fig:front} for a depiction of the front, both in dimension $2$.
Our main result is the scaling limit of the front, where we introduce the scaling parameter $L$ in the $s$-coordinate.

\begin{figure}
\centering
\includegraphics[width = 0.45\textwidth]{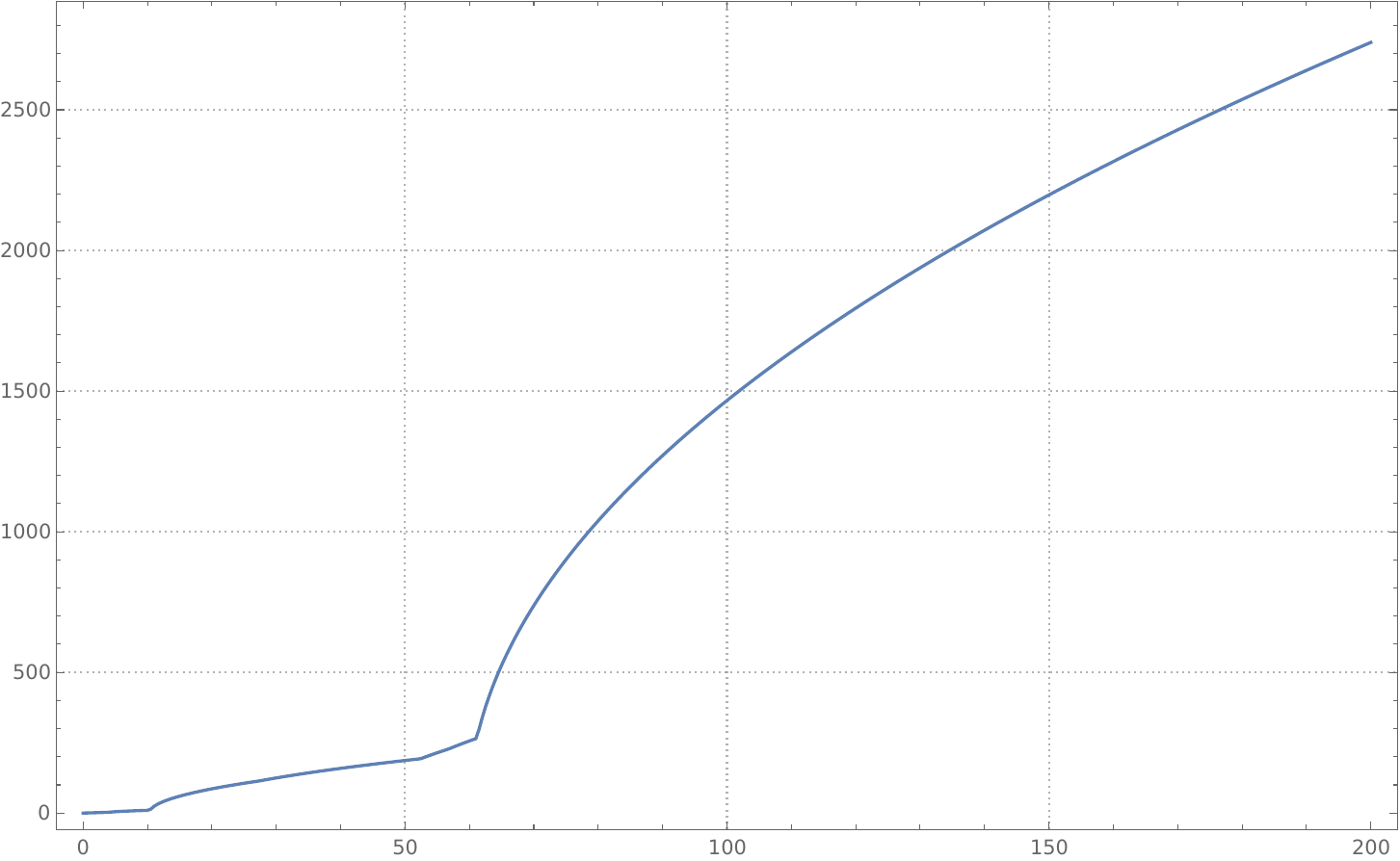}
\includegraphics[width = 0.45\textwidth]{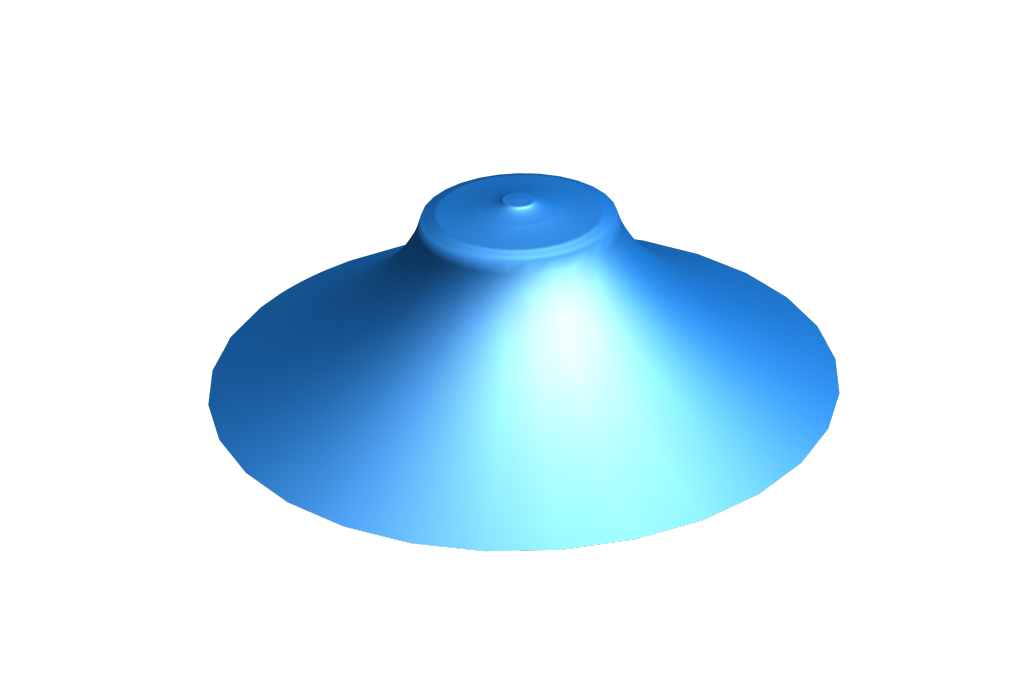}
\caption{(Left) A simulation of $\rho(s)$ for $s\in [0,200]$. (Right) The paraboloid formed by revolving $\rho(s)$ around an axis (the axis was not taken to be $\mathbf{e}_1$ for better resolution): this is a simulation of the scaling limit of the front $(\rho(s,\theta))_{s\in[0,200], \theta\in\S^{1}}$ of 3D BBM, rotated.}
\label{fig:rho-front-scaling-limit}
\end{figure}

\begin{maintheorem} \label{thm:extremal-front}
Fix $d\geq 2$ and $\ep \in (0,1)$. We have the following weak convergence in $D([0,\infty) \times \S^{d-2})$ (in the sense of finite-dimensional distributions and tightness in the Skorokhod space $D([0,T]\times \S^{d-2})$ equipped with the $\sup$ norm, for any fixed $T>0$): as first $t\to\infty$ then $L\to\infty$, 
\begin{align*}
  \big(  8^{-\frac{1}4}L^{-\frac{3}2} \mathfrak{h}_{\cC_t}^{\ep}(sL,\theta) \big)_{s \in [0,\infty),\theta \in \S^{d-2}} \Rightarrow \big( \rho(s,\theta) \big)_{s\in [0,\infty),\theta \in \S^{d-2}}\,,
\end{align*}
where 
\begin{equation}
  \label{eq-rho}
  \rho(s,\theta) := \rho(s) = \Big( \sup_{\sigma \geq 0} (\sigma s -\sigma R_\sigma) \Big)^{\frac12}\,, \text{ for all $s \geq 0$ and $\theta \in \S^{d-2}$,}
  \end{equation}
and $R_.$ denotes a Bessel(3) process started from $0$.
\end{maintheorem}
For clarity, the scaling limit is the 
surface 
$\rho(s,\theta)$ formed by revolving $\rho(s)$ around the $\mathbf{e}_1$-axis, see Figure~\ref{fig:rho-front-scaling-limit} for simulations. 
We prove \cref{thm:extremal-front} in \cref{sec:pf-main-thm}. Before proceeding, some remarks on \cref{thm:extremal-front} are in order.

\begin{remark}
Our proof will show that $\ep$ can be taken to $0$ as $L \to\infty$, so long as $\ep$ does not decay faster than $L^{-1/2}$. 
Similarly, our choice in \cref{def:front} of the interval $(-s, -s+1]$ having size $1$ was  arbitrary for our scaling limit: one could consider intervals of size $\delta$ tending to $0$ as $L \to\infty$ (again, with the condition that $\delta$ does not decay too fast).
\end{remark}

\begin{remark} \label{rk:properties-rho}
Here we make some observations on the limiting process $\rho$.
First, observe that $\rho^2(s)$ is a convex function, as it is a Fenchel-Legendre transform. Second, since  $R_0 = 0$ and $R_{\sigma}\to \infty$ almost surely, we have that, for any $s\geq 0$, the supremum defining $\rho(s)$ is achieved and $\rho(s)$ is positive almost surely.
Lastly, note that for $s>0$, $\rho(s)$ can be transformed as follows:
\[
\rho(s) = \Big(\sup_{\sigma\geq 0} (\sigma s- \sigma R_{\sigma})
\Big)^{\frac12} = \Big( \sup_{\sigma \geq 0} (\sigma s^3 - \sigma s^2 R_{\sigma s^2}) \Big)^{\frac12} = s^{\frac32} \Big( \sup_{\sigma \geq 0} (\sigma  - \sigma \tfrac{ R_{\sigma s^2}}{s}) \Big)^{\frac12}\,.
\]
Thus, $\rho(s)$ equals $s^{3/2} \rho(1)$ in distribution for each fixed $s>0$, by  Brownian scaling invariance of the Bessel process.
\end{remark}

\begin{remark}
  The exponent $3/2$ also appears in \cite{BZ18}
  as a one-sided bound for the
Brunet-Derrida $N$-BBM evolution. While that is a different model, and the
scaling is different, it is nevertheless striking that the same exponent 
occurs there. We also note that, some time after the posting of this paper, a similar-in-spirit variational problem was found in the fluctuations of $\cD[-v,0]$, where $\cD$ denotes the decoration measure in the extremal point process of 1D BBM \cite[Theorem~1.2]{HLW24}. Both here and there, the variational problem comes from optimizing over the branching times on the backwards trajectory of the spine particle: see the proof outline in \cref{sub:outline_of_the_proofs} for more.
\end{remark}

Our study of the front proceeds via an analysis of the distribution of the extreme values of multidimensional BBM. In particular, we obtain in \cref{thm:extremal-landscape} the distributional limit of the extremal cluster $\cC_t$, as $t\to\infty$, in a form that is explicit and amenable to analysis. It is via this analysis that we study the front, see \cref{thm:extremal-front-C}.
The distribution of the extremal particles in $d=1$
has been well-understood since the seminal
papers \cite{ABBS12,ABK13}, see also \cite{SBM21} and \cite{BKLMZ21}.
Considering multiple dimensions clearly adds a new level
of complexity due to important spatial considerations that complicate
and enrich
the extreme value theory compared to $d=1$ (though the $d=1$ setting has itself drawn much attention, in part due to its  role in the study of reaction-diffusion equations \cite{Julien,Bovier,Ryzhik}
and the extrema of log-correlated fields \cite{Z16,AR17,Bis20,BK22}).
In the following subsection, we discuss the development of the extreme value theory of multidimensional BBM up to the current article, focusing on connections with the study of the front.

\subsection{Extreme value theory of multidimensional BBM}
For $v \in \cN_t$ and $s\in [0,t]$, we denote the norm process $R_s(v) := \norm{\B_s(v)}$.

The first natural 
question in regards to the front is, ``where can we find the front?''; that is, ``at what distance from the origin is the frontier of the process located?'' 
This is equivalent to understanding the maximum norm amongst particles alive at time $t$, as $t$ grows to infinity.
As mentioned previously, in \cite{KLZ21}, the authors and Lubetzky proved that $R_t^* - m_t(d)$ 
converges in distribution to a randomly-shifted Gumbel random variable, where 
\[
    m_t(d) := \sqrt2t+ \tfrac{d-4}{2\sqrt2} \log t\,. 
\]
(In dimension $d=1$, this result has a long history, 
going back to \cite{bramson83,LS87} for the right-most particle, and \cite{SBM21}
for the joint law of the left-most and right-most particles.) 
We mention that the result was upgraded in \cite{BKLMZ21} to hold for BBM started from any initial configuration, and, as in the 1D case, the random shift was shown to arise from the early history of the BBM in the sense of Lalley-Sellke \cite{LS87}; see Corollary 4.3 and Proposition~1.3 there. Further, that work identified the random shift as the total mass of the almost-sure limit of a martingale sequence of measures. This limiting measure, denoted by $D_{\infty}(\theta)\sigma(\d\theta)$ for $\sigma(\d\theta)$ the Lebesgue measure on $\S^{d-1}$, was constructed in \cite{SBM21} as a higher-dimensional analogue of the so-called ``derivative martingale''.

As the front is built from the extremal particles, 
the next natural question is, ``how are the extremal particles  distributed?''
In \cite{BKLMZ21}, it was shown that the  collection of extremal particles, viewed as a point process on $\R\times \S^{d-1}$ (polar coordinates), converges weakly in the topology of vague convergence to a decorated Poisson point process of random intensity. Namely, we have convergence of the so-called  \emph{extremal point process}:
\begin{align} \label{eqn:epp-convergence}
\cE_t := \sum_{u \in \cN_t} \delta_{(R_t(u) - m_t(d), \theta_t(u))} \to \cE_{\infty}\,,
\end{align}
where $\cE_{\infty}$ takes the following description. Recall $D_{\infty}(\theta)\sigma(\d\theta)$ from above, and let $\{\cD^{(i)}\}_{i\in \N}$ be i.i.d.\ copies of the decoration point process $\cD$ for the 1D BBM (a description of $\cD$ is given in \cref{rk:one-dimensional-decoration}). 
Let $\{(\xi_i,\theta_i)\}_{i\in\N}$ denote the atoms of a Poisson point process on $\R\times \S^{d-1}$ with intensity 
\begin{align}
    \mathfrak{C}_d e^{-\sqrt2x}\d x \otimes D_{\infty}(\theta) \sigma(\d\theta)\,,
    \label{eqn:epp-intensity}
\end{align}
where $\mathfrak{C}_d>0$ is some constant. Then 
\begin{align}\label{def:cE-infinity}
    \cE_{\infty} = \sum_{i=1}^{\infty} \sum_{r \in \cD^{(i)}} \delta_{(\xi_i+r, \theta_i)}\,.
\end{align}
In $d=1$, the 
extremal point process $\sum_{u\in \cN_t} \delta_{B_t(u) - m_t(1)}$ was shown to converge to a decorated Poisson point process of random intensity in two independent and simultaneous works, \cite{ABBS12} and \cite{ABK13}.

Similar to the one-dimensional case, the structure of $\cE_{\infty}$ comes from an interesting genealogical description of the extremal particles. 
In \cite{KLZ21}, as a consequence of a modified second moment method, one can infer with high probability that all extremal particles branched from one another either before time $O(1)$ or after time $t-O(1)$.
In words, the extremal particles are related to one another as ``distant ancestors'' or ``close relatives''.
The Poisson points in $\cE_{\infty}$ may then be understood as the limiting locations  
of the extremal particles in $\cN_t$ of maximal norm amongst their close relatives (the ``clan leaders''), 
while each decoration corresponds to the point process of close relatives behind each clan leader. See \cref{def:E-hat} for a precise formulation of this statement.

However, notice that the decorations only appear in the norm component of $\cE_{\infty}$ --- there are no decorations in the angular component. This can be understood as a consequence of the fact that the angles in $\cE_t$ are measured with respect to the origin, and so the angular differences amongst the members of each clan are of order $t^{-1}$ (as their radial spread is $O(1)$), which of course vanishes in the limit. In particular, $\cE_{\infty}$ does not capture any \textit{local} information in the direction transversal to the radial component
of the extremal particles--- for instance, the limiting distribution of $\cC_t$ cannot be deduced--- and therefore
does not carry the information needed
to understand the front around $u^*$ or any other extremal point. 
This is addressed by our second main result, \cref{thm:extremal-landscape} below.

\subsection{The extremal landscape of multidimensional BBM} \label{subsec:extremal-landscape}
We introduce in \cref{def:extremal-landscape} a point process $\cE_{t,\ell}$ that recovers the local landscape around \textit{every} extremal point. For this reason, we call $\cE_{t,\ell}$ the \emph{extremal landscape}. The second main result of this article, \cref{thm:extremal-landscape}, gives the distributional limit of the extremal landscape. 

Before proceeding, we introduce the following.

\subsection*{Some genealogical notation}
For any two particles $u_1, u_2 \in \cN_t$, let $u_1\wedge u_2$ denote the last time at which the most recent common ancestor of $u_1$ and $u_2$ was alive (i.e., the time at which $u_1$ and $u_2$ split).
For $0<\ell<t$ and $u \in \cN_t$, define the \emph{$\ell$-clan} of $u$ to be 
$[u]_{\ell}:= \{v \in \cN_t: u \wedge v \geq t-\ell\}$. 
For $s \leq t$ and $u \in \cN_s$, let $\cN_t(u)$ denote the subset of $\cN_t$ formed by the descendants of $u$, and let $u_t^*$ denote the particle in $\cN_t(u)$ of maximal norm.
Denote the set of \emph{$\ell$-clan leaders} by
\begin{equation}
  \label{eq-lclanleader}
    \Gamma_{t,\ell} := \big\{u \in \cN_t:  R_t(u) = \max_{v \in [u]_{\ell}} R_t(v) \big\}= \{u_t^* : u \in \cN_{t-\ell} \}\,.
  \end{equation}

Recall the rotation map $\cR_{\theta}$ from \cref{def:front-bbm-extremal-cluster}.
\begin{definition}\label{def:extremal-landscape}
The \textit{extremal landscape} at time $t$ is the point process
\begin{align} 
    \cE_{t,\ell}^{\land} := \sum_{u\in \Gamma_{t,\ell}} \delta\Big(R_t(u) - m_t(d), \theta_t(u), \sum_{v \in \cN_t} \delta_{\cR_{\theta_t(u)}(\B_t(v) - \B_t(u))} \Big)\,.
\end{align}
\end{definition}
This point process is related in spirit to
 the extremal process of the two-dimensional discrete Gaussian free
field studied in \cite{BL18}. 
Now, let $\mathbb{M}$ denote the space of Radon measures on $\R^d$, and recall $\mathfrak{C}_d$ and $D_{\infty}(\theta)\sigma(\d\theta)$ from \cref{eqn:epp-intensity}.
\begin{maintheorem}\label{thm:extremal-landscape}
There exists a point process law $\nu$ on $\mathbb{M}$, described in Section \ref{subsec:extremallimitcluster},
such that the extremal landscape $\cE_{t,\ell}^{\land}$ converges weakly, as $t\to\infty$ then $\ell\to\infty$, to 
\[
    \cE_{\infty}^{\land}:= \mathrm{PPP}(\mathfrak{C}_d e^{-\sqrt2x} \d x \otimes D_{\infty}(\theta) \sigma (\d \theta) \otimes \nu)
 \]
 in the space of point processes on $\R \times \S^{d-1} \times \mathbb{M}$ endowed with the topology of vague convergence.
\end{maintheorem}
The proof of \cref{thm:extremal-landscape} is given in \cref{sec:proof-cluster}.
We note that the limiting point process $\cE_{\infty}^{\land}$ may be described as follows: let $\{(\xi_i, \theta_i)\}_{i\in\N}$ denote the atoms of $\mathrm{PPP}(\mathfrak{C}_d e^{-\sqrt2x} \d x \otimes D_{\infty}(\theta) \sigma (\d \theta))$, and let $\{\cC^{(i)}\}_{i\in \N}$ be i.i.d.\ according to $\nu$.  Then we have 
\begin{align} \label{eqn:extremal-landscape-rep}
    \cE_{\infty}^{\land} \egaldistr \sum_{i \in \N} \delta_{(\xi_i, \theta_i, \cC^{(i)})} \,.
\end{align}

The convergence of the extremal cluster $\cC_t$ (\cref{def:front-bbm-extremal-cluster}) follows immediately. 
\begin{maincoro}\label{cor:extremal-cluster}
As $t$ tends to $\infty$, the extremal cluster $\cC_t$ converges weakly in the topology of vague convergence to the point process $\cC$ with law $\nu$.
\end{maincoro}

\begin{remark}\label{rk:general-front}
  \cref{thm:extremal-front} describes the front around the \emph{maximal} particle $u^*$. As mentioned before, from \cref{thm:extremal-landscape} and our proof of \cref{thm:extremal-front}, it is evident that the fronts behind all $\ell$-clan leaders of near-maximal radial component  converge simultaneously to i.i.d.\ copies of $\rho$. 
\end{remark}

In the following subsection, we give an explicit description of $\cC$, which of course also describes $\nu$; see Figure~\ref{fig:just-extremal-cluster} for a simulation. The proof of \cref{thm:extremal-front} will proceed via the analysis of the front of $\cC$ (see \cref{thm:extremal-front-C}). 

\begin{figure}
\centering
\includegraphics[width = 0.9\textwidth]{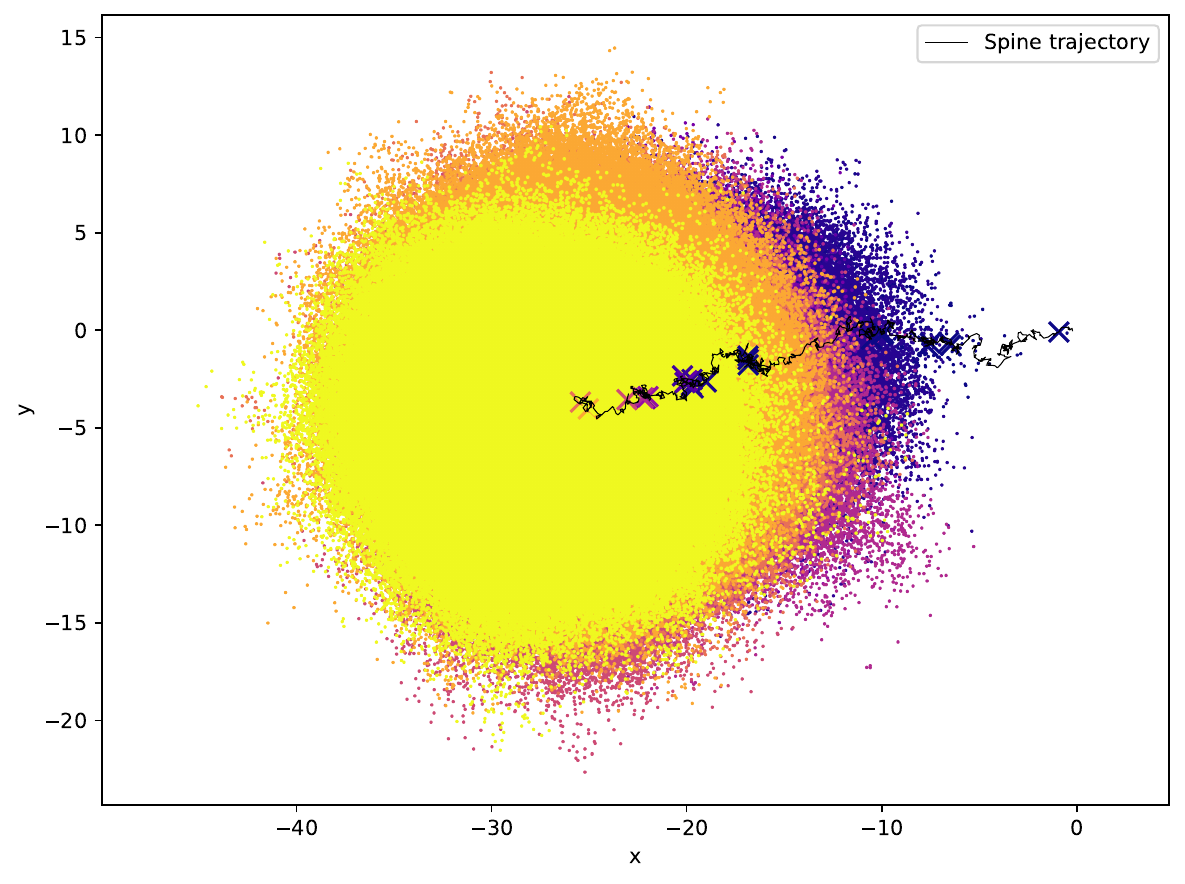}
\caption{A simulation of a point process approximating the extremal cluster $\cC$ in dimension $2$. The trajectory of the spine is approximated by $S_s := (-\sqrt2s-R_s, Y_s)$, where $R_s$ is a Bessel(3) process and $Y_s$ an independent standard Brownian motion (black). See \cref{rk:spine-bessel} for an explanation of why $A_s \approx -\sqrt2s-R_s$. 
The branching times are approximated by a Poisson point process of rate $2$, the location of the spine at each branching time is marked with a uniquely-colored ``$\times$'', and  the associated BBM point cloud is plotted in the same color. The simulation contains 25 BBM point clouds.
}
\label{fig:just-extremal-cluster}
\end{figure}

\subsection{Description of the limit of the extremal cluster}
\label{subsec:extremallimitcluster}
We now describe  the limiting point process $\cC$ with law $\nu$ via the following branching particle system, which we  refer to as $\mathfrak{B}$.
\begin{enumerate}
    \item (Spine) Define the \emph{spine} particle $\xi$ with path 
\[
    \big(\cS_s(\xi)\big)_{s\geq 0} = \big(A_s(\xi), \Y_s(\xi) \big)_{s \geq 0} \subset \R \times \R^{d-1}\,, 
\]
where $\Y_.(\xi)$ is a standard Brownian motion in $\R^{d-1}$ and $A_.(\xi)$ is an independent process in $\R$ whose exact definition is postponed to \cref{subsec:A-pi}, as it is not needed till then. 

    \item (Branching times) Next, conditionally on $A_.(\xi)$, define a Poisson point process $\pi = (\tau_1<\tau_2<\dots)$ on $[0,\infty)$ whose intensity is a function of $A_.(\xi)$ (we again postpone the precise definition of the intensity to \cref{subsec:A-pi}). We refer to the atoms of $\pi$ as \emph{branching times}.

    \item (BBM point clouds) At each time $\tau_i$, $\xi$ creates 
      a new particle, which generates
      a $d$-dimensional BBM $\mathfrak{B}^{i}$ started from $\cS_{\tau_i}$, stopped at time $2\tau_i$, and conditioned on the event $E_i$ defined as follows. 
    If $\cN_s^{i*}$ denotes the set of particles of $\mathfrak{B}^{i}$ alive at time $\tau_i+s$ and the trajectories of the particles $v\in \cN_s^{i*}$ are denoted by 
    \[
    \B^*_s(v) = (A_{\tau_i}(\xi) + X^*_s(v), \Y_{\tau_i}(\xi) + \Y^*_s(v)) \in \R \times \R^{d-1}\,,
    \]
    then we define the event 
\begin{align} \label{def:Ei}
    E_i := \{ \max_{v\in\cN_{\tau_i}^{i*}} X^*_{\tau_i}(v) + A_{\tau_i}(\xi) < 0\}\,.
\end{align}
Since $E_i$ is measurable with respect to the projection of $\mathfrak{B}^i$ onto the first standard basis vector $\mathbf{e}_1$, note that
the law of $\{\Y_s^*(v)\}_{v\in \cN_s^{i*}, s\in[0,\tau_i]}$ 
is that of a standard $(d-1)$-dimensional BBM conditioned on having the same genealogy as $\cN_s^{i*}$; in particular, the path $(\Y_s^*(v))_{s\in[0,\tau_i]}$ for each $v \in \cN_s^{i*}$ is a standard Brownian motion in $\R^{d-1}$ started from $0$.
We will refer to $\mathfrak{B}^{i}$ as the \emph{BBM point cloud born at time $\tau_i$}.
\end{enumerate}
Finally, we have
\begin{align} \label{def:extremal-cluster}
    \cC = \delta_0 + \sum_{i=1}^{\infty} \sum_{v \in \cN_{\tau_i}^{i*}} \delta_{\B^*_{\tau_i}(v)} \,.
\end{align}
We elaborate on the structure of $\cC$ and discuss connections with the 1D case in \cref{subsec:previous-results}.

\subsection{Outline of the proofs}
\label{sub:outline_of_the_proofs}
The statement of \cref{thm:extremal-landscape} is rather intuitive, given the knowledge of the extremal process from \cite{BKLMZ21} (\cref{eqn:epp-convergence} here):
indeed, the clan around an extremal 
particle 
can be constructed from 
the backwards trajectory of the extremal particle, viewed as ``the spine'',
together with collections of particles that have branched from the spine 
at times $t-s$, $s\geq 0$; this is described in 
\cref{subsec:previous-results}. The structure of the spine projected onto $\mathbf{e}_1$ is as 
in the case of one-dimensional BBM, and is given explicitly in \cite{ABBS12}.
The law  of the branched particles is that of a BBM conditioned not to
``overtake'' the spine particle at time $t$, 
and due to geometric considerations, it is not hard to see that the radial
increments of the process coincide with the direction of the clan leader at time $t-s$, for any fixed time $s$, as $t\to\infty$. This will show in turn
that the transversal increment of these particles is, given the genealogy of 
the particles, independent of the radial component. Altogether,
this will give \cref{thm:extremal-landscape}, and is explained 
in \cref{sec:proof-cluster}.

As mentioned earlier, we prove \cref{thm:extremal-front} as a consequence of \cref{thm:extremal-front-C}, which describes the front of $\cC$.
To analyze the front, we first identify
which particles in the description of $\cC$ have an $\mathbf{e}_1$-component that is likely to fall within the window $(-sL,-sL+1]$,
for each $s>0$ and $L$ large. 
One can see from a first moment computation that  most of these particles
have split from the spine at a time $\tau_i$ of order $L^2$; 
this fact is related to the analysis of the one-dimensional
decoration measure performed in \cite{CHL21}, see the paragraph below (1.27)
there. Moreover, these turn out to be the particles that constitute the front.
An important
factor affecting the number of particles landing in that window is the (radial)
location of the spine at those times --- given the spine, 
the maximal transversal displacement in any direction, being the maximum of an
exponential-in-$L$ 
number of Gaussian random variables of variance $L^2$, turns out to concentrate in the scale $L^{3/2}$.

Of course, these Gaussians are not independent, being correlated through the genealogy of the branching process, and our proof does not attempt to make rigorous the last step of the above heuristic. 
Our strategy for computing the maximal transversal displacement $M_s$ amongst particles in $\cC$ with $\mathbf{e}_1$-component inside the window $(-sL,-sL+1]$ is as follows. 
Recall from \cref{subsec:extremallimitcluster} that $\cC$ may be described as the concatenation of a sequence of \emph{BBM point clouds} along the trajectory of the \emph{spine}, which carries an explicit description. 
We can therefore compute $M_s$ by (1) computing the maximal and minimal extent of each BBM point cloud, (2) understanding  the fluctuations of the spine, and (3) optimizing over the BBM point clouds. Towards (1), in Lemma~\ref{lem:h-to-hbar}, we show that each BBM point cloud can be replaced with an independent, standard BBM run for the same amount of time, and moreover, only the BBM clouds born at time $\tau_i \approx L^2$ contribute to the front. The maximal extent of all of these BBMs can be bounded using known tail bounds on the maximum norm of BBM. A key insight is that the minimal extent can be computed using G\"{a}rtner's theorem
\cite{gartner}; we refer to the start of \cref{sec:proof-hbar-convergence} for a detailed outline of this argument. We remark that, in particular, we avoid the use of the second moment method, which is the usual but rather involved strategy employed to analyze the extrema of log-correlated fields.

Towards (2), we approximate the trajectory of the spine in terms of a Bessel(3) process  in \cref{sec:spine}. Ultimately, this is the Bessel process that appears in \cref{thm:extremal-front}. Optimizing over all BBM clouds  in step (3) results in the Fenchel-Legendre transform in \cref{thm:extremal-front}.

It then remains to analyze the spine fluctuations. 
There is a competition here, between the location of the 
spine (which itself, for large times, is close to a Bessel(3) process, see \cref{rk:spine-bessel})
and the number of branched particles
that land in the window. This competition, together with Brownian scaling
relations, leads to the variational problem
in  \eqref{eq-rho}.

\subsection{Asymptotic Notation}
\label{subsec:asymp-notation}
We use the standard (Bachmann-Landau) big $O$ and little $o$ notation. When the asymptotic parameter is unclear, we write it in the subscript; eg., if the asymptotic 
parameter is $K$, we will write $o_K(\cdot)$ and $O_K(\cdot)$.

\section{Distributional limit of the extremal landscape}
\label{sec:proof-cluster}
The description of the limiting extremal cluster $\cC$ (\cref{subsec:extremallimitcluster}) draws heavily from the description from \cite{ABBS12} of the 
decoration measure $\cD$ in the setting of one-dimensional BBM; indeed, $\cD$ is given by projecting $\cC$ onto the first standard basis vector $\mathbf{e}_1$.  
The Poisson point process structure of $\cE_{\infty}^{\land}$ is derived from the results in \cite{BKLMZ21} on multi-dimensional BBM. 
We recount the relevant results of \cite{ABBS12} and \cite{BKLMZ21} in \cref{subsec:previous-results}, then prove \cref{thm:extremal-landscape} in \cref{subsec:extremal-landscape-pf}.

\subsection{Preliminaries and previous results}
\label{subsec:previous-results}
The genealogical notation set out in \cref{subsec:extremal-landscape} will be used throughout the remainder of the section. 

We begin in the setting of \textit{one-dimensional} BBM, still denoted by $\B_t(v)$, and the results of \cite{ABBS12}. Consider the following point process on $(-\infty,0]$:
\[
    \cD_{t,\ell} = \sum_{v\in [u^*]_{\ell}} \delta_{\B_t(v) - \B_t(u^*)}\,,
\]
Fix $v\in [u^*]_{\ell}$, and let $v$ split from $u^*$ at time $t-\mathsf{t}_i$. From the formula
\[
    \B_t(v) - \B_t(u^*) = (\B_{t-\mathsf{t}_i}(u^*)- \B_t(u^*)) + (\B_t(v) - \B_{t-\mathsf{t}_i}(v)) \,,
\]
we understand 
$\B_t(v) - \B_t(u^*)$ as the backwards increment of $u^*$ accrued until the splitting time of $v$ and $u^*$, plus the increment accrued by $v$ after branching from $u$. 
In particular, it is helpful to view $\cN_t$ via what we call the \emph{spinal representation} (with $u^*$ as the spine):
begin at time $0$ with $u^*$ evolving under the backwards path 
$(\B_{t-s}(u^*)-\B_t(u^*))_{s\geq 0}$.
At each $\mathsf{t}_i$, $u^*$ creates a 
new particle, which in turns generates
the branching particle system $\mathfrak{N}^{(i)}$ from time $\mathsf{t}_i$ to $2\mathsf{t}_i$, after which this branching particle system is stopped (the particles stop moving and branching for the rest of time). 

Then $\cD_{t,\ell}$ is the point process formed by $0$ (coresponding to $v = u^*$) and the terminal locations of particles in the spinal representation of $\cN_t$ corresponding to branching times $\mathsf{t}_i \leq \ell$. In \cite[Theorem~2.3]{ABBS12}, it is shown  that $\cD_{t,\ell}$ converges
\textit{jointly} with the backwards path of $u^*$ to $(\mathcal{D}, A_.(\xi))$, recalling $A_.(\xi)$ from \cref{subsec:extremallimitcluster}). We remark that in our case, $\xi$ and its path $\cS_s(\xi)$ correspond in the prelimit to $u^* \in \cN_t$ and the backwards path of $u^*$; the Poisson point process $\pi$ corresponds to the branching times $\mathsf{t}_i$; 
and each conditioned BBM $\mathfrak{B}^i$ corresponds to $\mathfrak{N}^{(i)}$.

Though they do not need this fact, the work of \cite{ABBS12} also shows the convergence can be made jointly with the genealogy of the particles in the spinal representation, in the following sense. Enumerate  particles of $[u^*]_{\ell}$ in decreasing order 
(so $v_1 = u^*$ and $\B_t(v_i) \geq \B_t(v_{i+1})$). Form the matrix $\mathbf{M}_{t,\ell}[u^*] := [d(v_i, v_j)]$, where $d(v_i,v_j)$ denotes the branching time of $v_i$ and $v_j$ in the spinal representation. 
Similarly, enumerate the particles in $\mathfrak{D}$ in decreasing order, and form the (infinite) matrix $\mathbf{M}$ of branching times. 
Then $\cD_{t,\ell}$ converges jointly in distribution with $\mathbf{M}_{t,\ell}[u^*]$ to $(\cD, \mathbf{M})$. 

\begin{remark}\label{rk:one-dimensional-decoration}
Let $\mathfrak{D}$ be the  branching particle system on $\R$ formed by projecting $\mathfrak{B}$ (\cref{subsec:extremallimitcluster}) onto $\mathbf{e}_1$. Then $\cD$ is the point process of the terminal locations of the particles of $\mathfrak{D}$, and $\cC$ can be constructed from $\cD$ by associating independent Brownian motions in $\R^{d-1}$ to each segment of the genealogical tree of $\mathfrak{D}$ (note $\mathfrak{B}$ and $\mathfrak{D}$ share the same genealogical tree).
\end{remark}

We now relate
 \cref{thm:extremal-landscape} to the results of \cite{BKLMZ21}, where convergence of the extremal point process of \emph{multidimensional} BBM was proved, as in \cref{eqn:epp-convergence}.
As explained in \cite{BKLMZ21} (below their Proposition 1.3), one can infer the following two results from \cref{eqn:epp-convergence} and its proof. First, letting $\mu$ denote the law of $\cD$, we have the following weak convergence as first $t\to\infty$ then $\ell\to\infty$:
\begin{align} \label{def:E-hat}
    \hat{\cE}_{t,\ell} := \sum_{u \in \Gamma_{t,\ell}} \delta\Big(R_t(u) - m_t(d), \theta_t(u), \sum_{v \in [u]_{\ell}} \delta_{R_t(v) - R_t(u)} \Big) \to \hat{\cE}_{\infty}\,,
\end{align}
where $\hat{\cE}_{\infty} := 
\mathrm{PPP}( \mathfrak{C}_d e^{-\sqrt2x}\d x \otimes D_{\infty}(\theta)\sigma(\d\theta) \otimes \mu)$. 
Here, weak convergence takes place 
in the space of point processes on $\R \times \S^{d-1} \times \tilde{\mathbb{M}}$, where $\tilde{\mathbb{M}}$ denotes the space of Radon measures on $\R$, endowed with the vague topology.
Therefore, the contribution of \cref{thm:extremal-landscape} comes from the joint convergence with the third component of $\cE_{t,\ell}^{\land}$; in particular, we have added to the results of \cite{BKLMZ21} by supplying the information on the BBM around $\B_t(u)$ in directions transverse to $\theta_t(u)$, for each $u\in \Gamma_{t,\ell}$ that is extremal (has norm close to $m_t(d)$).
Second, it was mentioned in \cite{BKLMZ21} that one can add genealogical information to the convergence of $\hat{\cE}_{t,\ell}$, for instance, in the sense of \cite{BoHa17} and \cite{Mal18}. 
Joint convergence with the genealogy in the simpler sense above \cref{rk:one-dimensional-decoration} also holds: write the spinal representation of $[u]_{\ell}$ for each $u\in \Gamma_{t,\ell}$ (with $u$ as the spine), and form the matrix of branching times $\mathbf{M}_{t,\ell}[u]$. Then
\begin{align} \label{eqn:joint-convergence-genealogy}
    \cE_{t,\ell}^* := \sum_{u \in \Gamma_{t,\ell}} \delta\Big(R_t(u) - m_t(d), \theta_t(u), \sum_{v \in [u]_{\ell}} \delta_{R_t(v) - R_t(u)}, \mathbf{M}_{t,\ell}[u] \Big)\to \cE_{\infty}^*
\end{align}
in distribution, where $\cE_{\infty}^* := \mathrm{PPP}( \mathfrak{C}_d e^{-\sqrt2x}\d x \otimes D_{\infty}(\theta)\sigma(\d\theta) \otimes \mu \otimes \mathbf{m})$ and $\mathbf{m}$ denotes the law of $\mathbf{M}$ on $\R^{\Z}$. Using the notation of~\eqref{eqn:epp-intensity} and letting $\mathbf{M}^{(i)}$ denote i.i.d.\ copies of $\mathbf{M}$, we may express
\begin{align}\label{eqn:E-infinity-*}
    \cE_{\infty}^* = \sum_{i=1}^{\infty} \delta_{(\xi_i, \theta_i, \cD^{(i)}, \mathbf{M}^{(i)})}\,.
\end{align}

\subsection{Proof of \cref{thm:extremal-landscape}}
\label{subsec:extremal-landscape-pf}
We utilize the result of \cite{KLZ21} that the norm processes of extremal BBM particles lie in a restricted set of paths until time $t- \ell$, where $\ell$ is a parameter sent to $\infty$ after $t$. We call this restricted set of paths $\mathsf{ExtNorm}(t,\ell)\subset C[0,t-\ell]$, where $C[0,s]$ denotes the space of continuous functions from $[0,s]$ to $\R$. We will not need the full description, which is notation heavy;
in words,  $\mathsf{ExtNorm}(t-\ell)$ is the collection of paths that stay within an ``entropic region''  below a linear path up to time $t-\ell$, and end within a ``good window'' at time $t-\ell$   (more precisely, it is the set
$\mathsf{ExtNorm}(t-\ell) := \{ f\in C[0,t-\ell] :  \cB^{\Bumpeq}_{[0,t-L-\ell]}(f(\cdot+L)) \cap \{ \mathbf{y}(f(t-\ell)) \in [\ell^{1/3}, \ell^{2/3}] \}$,
where we used the notation $\mathcal{B}^{\Bumpeq}$ from \cite[(5.7)]{KLZ21} and $\mathbf{y}(\cdot)$ from \cite[(5.4)]{KLZ21}).

The only quantitative description of $\mathsf{ExtNorm}(t-\ell)$ we will need is the following uniform estimate on the endpoint of paths:
\begin{align}
  f\in\mathsf{ExtNorm}(t-\ell) \quad \mbox{\rm implies that} \quad |f(t-\ell) - m_t(d)| = O(\ell) \label{eqn:candidate-endpoint}
\end{align}
where the $O(\ell)$ holds uniformly over $f\in \mathsf{ExtNorm}(t,\ell)$ as $t\to\infty$, for $\ell$ large.

We call the subset of particles in $\cN_{t-\ell}$ whose norm processes lie in $\mathsf{ExtNorm}(t-\ell)$ the \emph{candidate particles}, as in, candidate to produce an extremal particle at time $t$ (see \cref{lem:banana}):
\[
    \cN_{t-\ell}^{\cand} := \{v\in \cN_{t-\ell} :  R_.(v) \in \mathsf{ExtNorm}(t,\ell)\}\,.
\]

For $\ell \leq t$ and $u \in \cN_{t-\ell}$,  let $u_{t,\ell}^* \in \cN_t$ denote the a.s.--unique descendant of $u$ in $\cN_t$ that travelled the furthest in the direction of $u$:
\[
\B_t(u_{t,\ell}^*) \cdot \theta_{t-\ell}(u) = \max_{ w \in \cN_t(u)} \B_t(w) \cdot \theta_{t-\ell}(u) \,.
\]
Lastly, define the point process
\begin{align}
    \cE_{t,\ell}^{\good} := \sum_{u \in \cN_{t-\ell}^{\cand}} \delta\Big(\B_t(u_{t,\ell}^*) \cdot \theta_{t-\ell}(u) - m_t(d), \theta_{t-\ell}(u), \sum_{v \in \cN_t(u)} \delta_{\cR_{\theta_{t-\ell}(u)}(\B_t(v) - \B_t(u_{t,\ell}^*))} \Big)\,.
    \label{eqn:lanscape-candidates}
\end{align}
The advantage of  $\cE_{t,\ell}^{\good}$ is that, for each $u\in \cN_{t-\ell}^{\cand}$, the evolution of its descendents over $[t-\ell, t]$ 
decomposes into two independent evolutions--- in $\mathrm{span}(\theta_{t-\ell}(u))$ and in the orthogonal complement--- conditionally on the genealogy of $\cN_t(u)$. Indeed, we have the following result.  

\begin{proposition}\label{prop:reduction-t-ell}
The following weak limits hold in the same sense as in \cref{thm:extremal-landscape}:
\begin{enumerate}
    \item the weak limit of $\cE_{t,\ell}^{\good}$ is $\cE_{\infty}^{\land}$, and
    \item the point process $\cE_{t,\ell}^{\land}$ has the same weak limit 
as $\cE_{t,\ell}^{\good}$.
\end{enumerate}
\end{proposition}

\begin{proof}[Proof of \cref{thm:extremal-landscape}]
\cref{thm:extremal-landscape} follows immediately from \cref{prop:reduction-t-ell}.
\end{proof}

We prove \cref{prop:reduction-t-ell}(1) below, and \cref{prop:reduction-t-ell}(2) in \cref{subsec:pf-prop-reduction-t-ell}.
\begin{proof}[Proof of \cref{prop:reduction-t-ell}(1)]
For each $v\in \cN_t(u)$, consider the projected, rotated increments:
\begin{align*}
    W_s^{i}(v) := \mathbf{e}_i \cdot \cR_{\theta_{t-\ell}(u)}(\B_{s+t-\ell}(v) - \B_{t-\ell}(v))  \quad \text{ for }  \text{$i \in \{1,\dots, d\}$} \text{ and } s \in [0,\ell]\,,
\end{align*}
where $\mathbf{e}_i$ denotes the $i^{\mathrm{th}}$ standard basis vector in $\R^d$.
Observe that $u_{t,\ell}^* = \mathrm{argmax}_{v\in \cN_t(u)} W_{\ell}^1(v)$.
Thus, the identification of $u_{t,\ell}^*$ is independent of the trajectory $(W_s^2(v),\dots,W_s^d(v))_{s\in [0,\ell]}$ for any fixed $v \in \cN_{t}(v)$; 
moreover, conditionally on the genealogy of $\cN_t(u)$ on the time interval $[t-\ell,t]$, the process 
\[
    \mathcal{W}^{\perp} := \{ (W_s^2(v), \dots, W_s^d(v))_{s\in [0,\ell]} : v \in \cN_t(u)\}
\]
is an independent $(d-1)$-dimensional BBM constructed on the genealogical tree of $\cN_t(u)$ (that is, for any edge of the genealogical tree of $\cN_t(u)$, the corresponding increment of $\mathcal{W}^{\perp}$ has the law of an independent $(d-1)$-dimensional Brownian motion).

This also implies that, conditionally on the genealogy of $\cN_t(u)$ on  $[t-\ell,t]$, $\mathcal{W}^{\perp}$ is independent of the point process $\cE_{t,\ell}^{\good,(1)}$ formed by projecting the third coordinate of $\cE_{t,\ell}^{\good}$  onto $\mathbf{e}_1$:
\[
    \cE_{t,\ell}^{\good,(1)}:= \sum_{u \in \cN_{t-\ell}^{\cand}} \delta\Big(\B_t(u_{t,\ell}^*).\theta_{t-\ell}(u) - m_t(d), \, \theta_{t-\ell}(u), \sum_{v \in \cN_t(u)} \delta_{\mathbf{e}_1\cdot\cR_{\theta_{t-\ell}(u)}(\B_t(v) - \B_t(u_{t,\ell}^*))} \Big)\,.
\]
Now, as a consequence of \cref{prop:reduction-t-ell}, the  point process $\cE_{t,\ell}^{\good,(1)}$ has the same weak limit (as first $t\to\infty$, then $\ell \to \infty$) 
as $\hat{\cE}_{t,\ell}$, defined in \cref{def:E-hat} with the limiting point process being
\[
    \hat{\cE}_{\infty} = \mathrm{PPP}( \mathfrak{C}_d e^{-\sqrt2x}\d x \otimes D_{\infty}(\theta)\sigma(\d\theta) \otimes \cD)\,.
\]
Moreover, in the sense of \cref{eqn:joint-convergence-genealogy,eqn:E-infinity-*}, $\cE_{t,\ell}^{\good,(1)}$ converges jointly with the genealogy of $\cN_t(u)$ on $[t-\ell,t]$ for each $u \in \cN_{t-\ell}^{\cand}$ 
to $\hat{\cE}_{\infty}$ with i.i.d.\ copies of $\mathbf{M}$.

It is then immediate from the aforementioned conditional independence of $\mathcal{W}^{\perp}$ and the discussion in \cref{rk:one-dimensional-decoration} 
on the connection between $\cC$ and $\cD$
that the distributional limit of $\cE_{t,\ell}^{\good}$ is obtained from $\hat{\cE}_{\infty}$ by replacing the law of $\cD$ with the law of $\cC$ (i.e., $\mu$ with $\nu$). The resulting point process is exactly $\cE_{\infty}^{\land}$.   
\end{proof}

\subsection{Proof of \cref{prop:reduction-t-ell}(2)} 
\label{subsec:pf-prop-reduction-t-ell}
The first input,
\cref{lem:banana}, shows that candidate particles are the only particles in $\cN_{t-\ell}$ that can produce extremal particles in $\cN_t$.

\begin{lemma}[{\cite{KLZ21}}]
\label{lem:banana}
Fix $K>0$. With $\mathsf{ExtNorm}(t,\ell)$ defined at the start of \cref{subsec:extremal-landscape-pf}, we have
\begin{align} \label{eqn:banana}
    \lim_{\ell\to\infty} \liminf_{t\to\infty} \P\big(\exists u\in \cN_{t-\ell} \setminus \cN_{t-\ell}^{\cand}\,,\, \exists v\in \cN_t(u) : R_t(v) - m_t(d) \geq -K \big) = 0\,. 
\end{align}
Moreover, there exists a function $f: \R_+\to\R_+$ such that
\begin{align} \label{eqn:bounded-number-candidate}
\lim_{\ell \to \infty}\limsup_{t\to\infty}\P\big(\#\big\{v \in \cN_t(u) : u \in \cN_{t-\ell}^{\cand} \big\} \leq f(\ell)\Big)  =1\,.
\end{align}
\begin{proof}
The lemma is a consequence of \cite[Theorem~3.1, Lemma~4.3, Lemma~5.1, and Claim~5.5]{KLZ21} and their proofs. To clarify, Theorem~3.1 of \cite{KLZ21} shows that, with high probability, any particle $v\in \cN_t$ such that $R_t(v) - m_t(d) \geq -K$ must satisfy the time-$L$ condition $z:= \sqrt2L - R_L(v) \in [L^{1/6},L^{2/3}]$, where $L:= L(\ell)$ can be any function satisfying $L\geq \ell^{1/6}$.  Lemma~4.3, Lemma~5.1, and Claim~5.5 of \cite{KLZ21}  consider BBM started from a particle with norm $\sqrt2L-z$ at time $0$ (instead of time $L$, in light of the Markov property), for $z$ satisfying the above condition. 
Along with the proof of \cite[Lemma~5.1]{KLZ21}, these results show the claim.
\end{proof}
\end{lemma}
\cref{eqn:bounded-number-candidate} demonstrates the advantage of restricting to $\cN_{t-\ell}^{\cand}$: to study the extremal particles, it is sufficient to consider $f(\ell)$ particles instead of $|\cN_t| \approx e^t$ particles. We apply this in the following. 

\begin{lemma}\label{lem:radius-of-cluster}
For any $\ell >0$, there exists a function $g: \R_+\to \R_+$ such that 
\begin{align*}
    \lim_{\ell\to\infty}\liminf_{t\to\infty} \P\Big(\sup_{u \in \cN_{t-\ell}^{\cand}} \sup_{v\in \cN_t(u)} \| \B_t(v) - \B_{t-\ell}(v) \| \leq g(\ell) \Big) = 1\,.
\end{align*}
\begin{proof}
For $g(\ell)$ sufficiently large compared to $f(\ell)$, this is an immediate consequence of \cref{eqn:bounded-number-candidate}, a union bound over at most $f(\ell)$ particles, and a Gaussian tail bound on $\| \B_t(v) - \B_{t-\ell}(v) \|$.
\end{proof}
\end{lemma}
Recall that for any $u\in\cN_{t-\ell}$ and $v\in \cN_t(u)$,  we have $\B_{t-\ell}(v) = \B_{t-\ell}(u)$.
\cref{lem:radius-of-cluster} states that, for all $u\in \cN_{t-\ell}^{\cand}$, the radial spread of its time-$t$ descendants $\cN_t(u)$ is $O(1)$ with high probability. This has two useful consequences. First, in light of \cref{eqn:candidate-endpoint},  candidate particles are at distance $\sqrt2t(1+o(1))$ from the origin at time $t-\ell$, and thus the angular spread of the time-$t$ descendents of a candidate particle is $O(t^{-1})$; this is the content of \cref{cor:angular-difference}.

\begin{corollary}\label{cor:angular-difference}
For $g(\ell)$ as in \cref{lem:radius-of-cluster}, we have
\begin{align*}
    \lim_{\ell\to\infty} \liminf_{t\to\infty} \P\big(\sup_{u \in \cN_{t-\ell}^{\cand}} \sup_{v\in \cN_t(u)} \| \theta_t(v) - \theta_{t-\ell}(u) \| \leq g(\ell)t^{-1} \big) = 1\,.
\end{align*}
\begin{proof}
This follows from \cref{eqn:candidate-endpoint}, \cref{lem:radius-of-cluster}, and the formula $\arctan(x)/x\to 1$ for $x\to 0$.
\end{proof}
\end{corollary}

The second consequence of \cref{lem:radius-of-cluster} is \cref{cor:one-clan}: for every $\ell$-clan leader $u_t^*$ (see \eqref{eq-lclanleader}) that is extremal, every particle that lies within $O(1)$ distance of $\B_t(u_t^*)$ must be a member of the $\ell$-clan $[u_t^*]_{\ell}$, with high probability.

\begin{corollary}\label{cor:one-clan}
For any $K_1, K_2>0$, we have 
\[
    \lim_{\ell\to\infty} \limsup_{t\to\infty} \P \big( 
\exists u \in \cN_{t-\ell}^{\cand} \,,\, \exists v \in \cN_t \setminus \cN_t(u) : R_t(u_t^*) - m_t(d) \geq -K_1\,,\,  \norm{\B_t(u_t^*) - \B_t(v)} \leq K_2 \big) = 0\,.
\]
\begin{proof}
Fix positive constants $K_1$ and $K_2$, as well as $u \in \cN_{t-\ell}^{\cand}$. Note that 
\[
    \{ R_t(u_t^*) - m_t(d)\geq - K_1 \,,\,  \norm{\B_t(u_t^*) -\B_t(v)} \leq K_2 \} \subset \{ R_t(v) - m_t(d) \geq -K_1-K_2 \}\,.
\] 
Recall from \cref{def:E-hat} that $\hat{\cE}_{t,\ell}$ converges to $\hat{\cE}_{\infty}$ and the intensity measure $D_{\infty}(\theta)\sigma(\d \theta)$ of the second component of  $\hat{\cE}_{\infty}$ has no atoms. This implies that no two clan leaders (particles in $\Gamma_{t,\ell}$) of size comparable to $m_t(d)$ can have the same angle in the limit. In particular,
for $v \in \cN_t\setminus \cN_t(u)$ satisfying
\[
     R_t(v) - m_t(d) \geq -K_1-K_2 \quad \text{ and } \quad R_t(u_t^*) - m_t(d) \geq K_1 \,,
\]
we have $\|\theta_t(v) - \theta_t(u_t^*)\| \geq \epsilon$ with probability tending to $1$ as first $t\to\infty$, then $\ell\to\infty$ and $\epsilon\to 0$. \cref{cor:one-clan} then follows from \cref{cor:angular-difference}.
\end{proof}
\end{corollary}

The following lemma and its corollary allow us to view the norm of each particle as indistinguishable from its projection onto its angle at time $t-\ell$.
\begin{lemma}\label{lem:norm-same-t-ell}
There exists $h:\R_+\to\R_+$ such that
\begin{align*}
    \lim_{\ell \to\infty}\liminf_{t\to\infty} \P\Big(\sup_{u \in \cN_{t-\ell}^{\cand}} \sup_{v\in \cN_t(u)} \| R_t(v) - \B_t(v)\cdot\theta_{t-\ell}(u) \| \leq 
    h(\ell) t^{-1}\Big) = 1\,.
\end{align*}
\begin{proof}[Proof of \cref{lem:norm-same-t-ell}]
Fix $u \in \cN_{t-\ell}^{\cand}$ and $v\in \cN_t(u)$. 
The idea is simply that the motion of $\B_.(v)$ on $[t-\ell,t]$ in the plane transverse to $\theta_{t-\ell}(v)$ will have no impact on the size of the norm. 

Towards this end, recall the projection operator $P^{\perp}_{\theta} := \Id - \theta \theta^t$, for $\theta\in \S^{d-1}$. 
Define the increment $\bB_{\ell}(v) := \B_t(v) - \B_{t-\ell}(u)$ and recall $\B_{t-\ell}(u) = \B_{t-\ell}(v)$. Then we compute
\begin{align*}
    R_t(v) = \big( (\B_t(v)\cdot\theta_{t-\ell}(u))^2 + \|P_{\theta_{t-\ell}(u)}^{\perp}(\bB_{\ell}(v))\|^2 \big)^{1/2} 
    \leq \big( (\B_t(v)\cdot\theta_{t-\ell}(u))^2 + \|\bB_{\ell}(v)\|^2\big)^{1/2}\,.
\end{align*}
Now, the formula $R_{t-\ell}(u) = \B_{t-\ell}(u)\cdot \theta_{t-\ell}(u)$, \cref{eqn:candidate-endpoint}, and \cref{lem:radius-of-cluster} imply the existence of $h:\R_+\to\R_+$ such that, for all $u \in \cN_{t-\ell}^{\cand}$ and $v\in \cN_t(u)$,
\[ 
|\B_t(v) \cdot \theta_{t-\ell}(u) - m_t(d)| \leq h(\ell) \quad \text{ and } \quad \|\bB_{\ell}(u)\|\in [0, h(\ell)]
\]
hold with high probability  (i.e., probability tending to $1$ as $t\to\infty$). Then, with high probability,
\[
    R_t(v)\leq \B_t(v) \cdot \theta_{t-\ell}(u) \sqrt{1 + \tfrac{h(\ell)}{(\B_t(v) \cdot\theta_{t-\ell}(u))^2}} \leq \B_t(v) \cdot \theta_{t-\ell}(u) + h(\ell)t^{-1}
\]
for all $u \in \cN_{t-\ell}^{\cand}$ and $v\in \cN_t$. Since $R_t(v) \geq \B_t(v) \cdot\theta_{t-\ell}(u)$, this completes the proof.
\end{proof}
\end{lemma}

\begin{corollary}\label{cor:ut=utl} Recall $u_t^*$ from the genealogical notation in \cref{subsec:extremal-landscape}.
For any $\ell>0$,
\begin{align*}
\lim_{t\to\infty} \P( \forall u \in \cN_{t-\ell}^{\cand} : u_t^* = u_{t,\ell}^*) = 1\,.
\end{align*}
\begin{proof}[Proof of \cref{cor:ut=utl}]
On the event described in \cref{lem:norm-same-t-ell},  the 
 event $\{\exists u \in \cN_{t-\ell}^{\cand} : u_t^* \neq u_{t,\ell}^*\}$ implies that $R_t(u_t^*)- R_t(u_{t,\ell}) \in [0,h(\ell)t^{-1}]$ for some $u\in \cN_{t-\ell}^{\cand}$. Thus, it suffices to show the following quantity vanishes to $0$ as $t\to\infty$:
\begin{align*}
    \P\big(\exists u \in \cN_{t-\ell}^{\cand}, ~\exists v_1, v_2 \in \cN_t(u) : |R_t(v_1)-R_t(v_2)| \leq h(\ell)t^{-1} \big)
\end{align*}
This follows from a union bound and an upper bound on 
the probability that a  Bessel process at time $\ell$ lies in an arbitrary ball of radius $O(t^{-1})$.
\end{proof}
\end{corollary}

With these inputs, we can finish the proof of \cref{prop:reduction-t-ell}.

\begin{proof}[Proof of \cref{prop:reduction-t-ell}(2)]
\cref{cor:ut=utl} allows us to replace all instances of  $u_{t,\ell}^*$ in $\cE_{t,\ell}^{\good}$ with $u_t^*$, without changing the weak limit.
\cref{lem:norm-same-t-ell} allows us to replace  $\B_t(u_t^*)\cdot \theta_{t-\ell}(u)$ in the first coordinate with $R_t(u_t^*)$. 
\cref{cor:angular-difference} allows us to replace $\theta_{t-\ell}(u)$ in the second coordinate of  $\cE_{t,\ell}^{\good}$ with $\theta_t(u_t^*)$.
\cref{lem:radius-of-cluster} and \cref{cor:angular-difference} together allow us to replace the rotation $\cR_{\theta_{t-\ell}(u)}$ in the third coordinate of $\cE_{t,\ell}^{\good}$ with $\cR_{\theta_t(u_t^*)}$.
\cref{cor:one-clan} allows us to replace the sum over $v\in \cN_t(u)$ in the third coordinate of $\cE_{t,\ell}^{\good}$ with a sum over $v \in \cN_t$. Finally, \cref{eqn:banana} allows us to replace the sum over $u\in \cN_{t-\ell}^{\cand}$ in $\cE_{t,\ell}^{\good}$ with a sum over $u\in \Gamma_{t,\ell}$. The resulting point process after all of these replacements is exactly $\cE_{t,\ell}^{\land}$.
\end{proof}

\section{Proof of \cref{thm:extremal-front}}
\label{sec:pf-main-thm}
In the statement of \cref{thm:extremal-front}, $L$ is taken to infinity  after
$t$. Therefore, it suffices to compute the scaling limit of the front of the limiting extremal cluster $\mathcal{C}$ in lieu of the front of $\mathcal{C}_t$.
\begin{maintheorem}\label{thm:extremal-front-C}
Fix $d\geq 2$ and $\ep \in (0,1)$, and recall $\rho(s,\theta)$ from \cref{thm:extremal-front}. We have the following weak convergence on $D([0,T) \times \S^{d-2})$ equipped with the $\mathrm{sup}$ norm, for any $T>0$, as $L\to\infty$:
\begin{align*}
  \big(  8^{-\frac{1}4}L^{-\frac{3}2} \mathfrak{h}_{\cC}^{\ep}(sL,\theta) \big)_{s \in [0,\infty),\theta \in \S^{d-2}} \Rightarrow \big( \rho(s,\theta) \big)_{s\in [0,\infty),\theta \in \S^{d-2}}\,.
\end{align*}
\end{maintheorem}

\begin{proof}[Proof of \cref{thm:extremal-front}]
\cref{thm:extremal-front} is a consequence of \cref{thm:extremal-landscape} and \cref{thm:extremal-front-C}.
\end{proof}

In \cref{subsec:A-pi}, we supply the full definitions of $A_.(\xi)$ and $\pi$ that were omitted from the description of $\cC$ from \cref{subsec:extremallimitcluster}. 
In \cref{subsec:proof-thm-extremalfrontC},  we prove \cref{thm:extremal-front-C} as an immediate consequence of two results: \cref{prop:sup-norm,prop:X_L-rho-convergence}. \cref{prop:sup-norm} is proved in \cref{subsec:pf-prop-sup-norm} by further dividing the result into \cref{lem:h-to-hbar,lem:hbar-convergence}. The proof of these lemmas
forms the technical heart of the article, and appears in Sections \ref{sec:simplifications} and \ref{sec:proof-hbar-convergence},
respectively. \cref{prop:X_L-rho-convergence} is proved in \cref{sec:spine}.

\subsection{The trajectory of the spine particle and the branching times} \label{subsec:A-pi}
The spine trajectory $A_.(\xi)$ and the Poisson point process of branching times $\pi$ come from the description by \cite{ABBS12} of the one-dimensional decoration point process $\cD$. 
We construct all objects  here and in \cref{subsec:extremallimitcluster} on a global probability space, denoted by $(\Omega, \cF, \P)$.
 
Define a family of processes indexed by the real parameter $b >0$ as follows. Let $R_.$ be a Bessel(3) process with $R_0 = 0$, and let $B_.$ be an independent standard Brownian motion in $\R$. Define $T_b^B := \inf \{t\geq 0 : B_t = b\}$,  the first hitting time of $b$ by $B_.$. Then, for each $b>0$,  define 
        \begin{equation} \label{def:Gamma}
            \Gamma_s^{(b)} := \begin{cases}
                B_s \,, &\text{for } s \in [0, T_b^B]\,, \\
                b- R_{s-T_b^B} &\text{for } s\geq T_b^B\,.
            \end{cases}
        \end{equation}
Let $M_r$ denote an independent random variable having the law of the maximum of a standard one-dimensional BBM at time $r$.
Define 
    \[
        G_r(x) := \P(M_r \geq \sqrt{2}r - x/\sqrt{2}) \,.
    \]
The law of $A_.(\xi)$ is given by the following formula\footnote{On page 411 of \cite{ABBS12}, this law appears
with an errant pre-factor of $\sigma := \sqrt2$ in front of the $\sup_{s\geq 0}(A_s+\sqrt2 s)$ term. This would imply that the supremum of $A_s+\sqrt2s$ is analogous to the the supremum of $\Gamma_s^{(b)}-\sqrt{2}s$ divided by $\sqrt2$, which easily can be seen to yield a contradiction.}: for any measurable  $E \subset C([0,\infty))$ and $b\geq 0$, 
        \begin{align}
            \P\big(A_.(\xi) \in E \,,\, \sup_{s\geq 0} \, (A_s(\xi)+\sqrt2 s) \in \d b\big)
            = \frac{1}{c_1} \E\Big[e^{-2 \int_0^{\infty} G_r(\sqrt2 \Gamma_r^{(b)}) \d r} \ind{(\Gamma_s^{(b)}- \sqrt2 s)_{s\geq 0} \in E} \Big]\,, \label{eqn:law-A}
        \end{align}
        where 
        \[ 
            c_1:= \int_0^{\infty} \E[e^{-2 \int_0^{\infty} G_r(\sqrt2 \Gamma_r^{(b)}) \d r}] \d b \,.
        \]
        The constant $c_1$ was shown to be finite in \cite[Equation~(6.7)]{ABBS12}. Note that 
        \[
            A^*(\xi) := \sup_{s\geq 0} \, (A_s(\xi)+\sqrt2 s)
        \]
        is a random variable taking values in $(0,\infty)$ with density
        \begin{align}
            \P( A^*(\xi) \in \d b) = f_{A^*(\xi)}(b)\d b  \,, \quad \text{ where } \quad f_{A^*(\xi)}(b) := \frac{1}{c_1}\E\Big[e^{-2 \int_0^{\infty} G_r(\sqrt2 \Gamma_r^{(b)}) \d r} \Big] \,. 
            \label{eqn:max-density}
        \end{align}
Finally, conditionally on $A_.(\xi)$, the intensity measure of  $\pi$ is 
    $
    2 \P_{A_t}(M_t <0) \d t
    $, for $t\geq 0$. 

\begin{remark} \label{rk:spine-bessel}
We will show that, as a consequence of exponential right-tail decay of $M_r$ (\cref{eqn:mallein}) and the Bessel behavior of $\Gamma^{(b)}$ (which goes to $-\infty$ at square-root speed), the Radon-Nikodym derivative in \eqref{eqn:law-A} has essentially no effect after sufficiently long time. In particular, $A_.(\xi)$ conditioned on $\{\sup_{s\geq 0} A_s(\xi) + \sqrt2s \in \d b\}$ behaves like a negative Bessel(3) process with drift $-\sqrt2$ after some time (this is made rigorous in  \cref{lem:spine-bessel-convergence} and its proof).
\end{remark}

\begin{remark}
The explicit formula for the intensity of $\pi$ is inessential to our analysis: we will only use that the intensity is bounded away from $0$ and above by $2$.
\end{remark}

\subsection{Proof of \cref{thm:extremal-front-C}}
\label{subsec:proof-thm-extremalfrontC}
Throughout the rest of this paper, we fix $\ep \in (0,1)$ and a time-horizon $T>0$. We also condition on $\{A^*(\xi) \in \d b\}$: due to the bounded convergence theorem, it suffices to prove convergence of finite dimensional distributions  and tightness on $[0,T]$ conditioned on this event.
Define
\begin{align}
    \Pb (\cdot) := \P( \cdot \given A^*(\xi) \in \d b)\,, \label{def:Pb} 
\end{align}
where we recall the global probability space $(\Omega,\cF, \P)$ from \cref{subsec:A-pi}.
Recall $c_1$ and $f_{A^*(\xi)}(b)$ from \cref{eqn:max-density}, and define $Z(b):= c_1 f_{A^*(\xi)}(b)$. 
Define the process   
\begin{equation}\label{eqn:conditioned-law-A}
    \hat{A}_s(\xi) := -A_s(\xi) - \sqrt{2}s\,,\, \quad 
    \Pb(\bA_.(\xi) \in E) = \frac{1}{Z(b)}  \E\Big[e^{-2 \int_0^{\infty} G_r(\sqrt2 \Gamma_r^{(b)}) \d r} \ind{(-\Gamma_s^{(b)})_{s\geq 0} \in E} \Big]\,.
\end{equation}
\cref{prop:sup-norm} states that $\h$ is essentially a deterministic functional of the spinal trajectory $\hat{A}(\xi)$, while \cref{prop:X_L-rho-convergence} states that this functional converges to the limit process $\rho(s)$ from \eqref{eq-rho}. 
\begin{proposition}\label{prop:sup-norm}
    Define the process
    \begin{equation}
    \label{eq-XLdef}
    X_L(s):= \Big(\sup_{\sigma \geq0} (\sigma s - \sigma \frac{\bA_{\sigma L^2}(\xi)}{L})\Big)^{\frac{1}2}\,.
  \end{equation}
    There exists a coupling of $\h$ and $X_L$ (with respect to their laws under $\Pb$) such that 
    \[
    \lim_{L\to\infty}\sup_{s\in[0,T],\theta \in \S^{d-2}}\|8^{-\frac14}L^{-\frac32}\h(sL,\theta) - X_L(s)\| = 0 \text{ in probability.}
    \]
\end{proposition}

\begin{proposition} \label{prop:X_L-rho-convergence}
In the space  $(C[0,T], \|\cdot\|_{\infty})$, the process $(X_L(s))_{s \in[0,T]}$ under $\Pb$ converges weakly 
to $(\rho(s))_{s\in[0,T]}$ as $L\to\infty$.
\end{proposition}
\begin{proof}[Proof of \cref{thm:extremal-front-C}]
\cref{thm:extremal-front-C} is an immediate consequence of \cref{prop:sup-norm,prop:X_L-rho-convergence}.
\end{proof}

Below, we prove \cref{prop:sup-norm} as a consequence of \cref{lem:h-to-hbar,lem:hbar-convergence}.
\cref{prop:X_L-rho-convergence} is proved in \cref{sec:spine} as a consequence of the Bessel-like nature of $\hat{A}_{\sigma L^2}(\xi)/L$ (c.f.\ \cref{rk:spine-bessel}).

\subsection{Proof of \cref{prop:sup-norm}}
\label{subsec:pf-prop-sup-norm}
The first result we need, \cref{lem:h-to-hbar}, shows 
$L^{-3/2}\h(sL,\theta)$ has the same asymptotic law, as $L\to\infty$,
as a simplified object $\hs(s,\theta)$,
defined as follows.

First, observe we may write $\h$ as a ``double maximum'': a maximum over the atoms of $\pi$, indexed by $i \in \N$, of the maximum transversal displacement amongst  particles in the $i^{\mathrm{th}}$ conditioned BBM cloud $\cN_{\tau_i}^{i*}$. More precisely, for each $i \in \N$, define
\[
    Z_L^{i*}(s, \theta) := \max_{v \in \cN_{\tau_i}^{i*}} \| \Y_{\tau_i}^*(v) + \Y_{\tau_i}(\xi) \| \ind{X_{\tau_i}^*(v) + A_{\tau_i}(\xi) \in (-sL, -sL+1]} \ind{|\theta \cdot \arg (\Y_{\tau_i}^*(v)+\Y_{\tau_i}(\xi))| \geq 1- \ep} \,.
\]
Then, for each $L>0$, $s\geq 0$, and $\theta\in\S^{d-2}$, we have
\begin{align}
    \h(sL,\theta) = \max_{i \in \N}  Z_L^{i*}(s, \theta)\,. \label{eqn:extremal-front-double-maximum}
\end{align}
We sometimes refer to $Z_L^{i*}$ as the contribution of the $i^{\mathrm{th}}$ conditioned BBM cloud $\cN_{\tau_i}^{i*}$  to  $\h$.

Now, great simplifications can be made by localizing the set of $i\in \N$ for which $Z_L^{i*}$ contribute to $L^{-3/2}\h(sL,\theta)$. Moreover, it turns out that replacing the conditioned BBM clouds with independent, standard BBM does not affect the limit. Towards this end,
for each atom $\tau_i$ of $\pi$, generate an independent, standard $d$-dimensional BBM $\{\B_s(v) \in\R^d : v\in \cN_s^{i}\}_{s\geq0}$ on the same probability space, where  $\B_s(v) := (  X_s(v), \Y_s(v)) \in \R\times \R^{d-1}$. For each $i \in \N$, $s\geq 0$, and $\theta \in \S^{d-2}$, define
\begin{align}
    Z_L^i(s, \theta) := \max_{v \in \cN_{\tau_i}^{i}} \|\Y_{\tau_i}(v)\| \ind{X_{\tau_i}(v) +A_{\tau_i}(\xi)\in (-sL,-sL+1]}\ind{|\theta \cdot \arg (\Y_{\tau_i}(v)+\Y_{\tau_i}(\xi))| \geq 1-\ep}  \,,
    \label{def:ZLi}
\end{align}
Finally, recall the time horizon $T$, and define  the \emph{simplified front}
\begin{align}
    \hs(s,\theta) :=  L^{-\frac32} \max_{i\in\N} Z_{L}^{i}(s,\theta) \ind{\tau_i \in [
    L^{1.4}, \tau_{TL}(\bA(\xi))]} \,, \label{def:simplified-front}
\end{align}
where used the following notation for the last exit time of a stochastic process $\mathcal{X}_.$: for $C>0$, define
\begin{equation}\label{def:last-exit-time}
\tau_C(\mathcal{X}) := \sup_{t\geq 0} \{ \mathcal{X}_t \in [0,C] \}\,.
\end{equation}

\begin{lemma} \label{lem:h-to-hbar}
There exists a coupling of the pair $\h$ and $\hs$ in the following sense: there exist random variables $\hcirc$ and $\hsb$ defined on a probability space $(\mathsf{\Omega}, \mathsf{F}, \mathsf{P}^{(b)})$ such that
    $\h$ under $\Pb$ has the same law as $\hcirc$ under $\mathsf{P}^{(b)}$;
    $\hs$ under $\Pb$ has the same law as $\hsb$ under $\mathsf{P}^{(b)}$; and
\[
    \lim_{L\to\infty}\sup_{s\in[0,T],\theta \in \S^{d-2}}\|L^{-\frac32}\hcirc(sL,\theta) - \hsb(s,\theta)\| =0 \text{ in probability.}
\]
\end{lemma}

Observe that, compared to the definition of $Z_L^{i*}$, $Z_L^i$ does not have the $+\mathbf{Y}_{\tau_i}(\xi)$ term. This omission is possible because $L^{-3/2}\mathbf{Y}_{\tau_i}(\xi)$ is negligible for $\tau_i \leq \tau_{TL}(\hat{A}(\xi))$, as the stopping time is of order $L^2$ so $|\mathbf{Y}_{\tau_i}(\xi)|$ is of order $L$; see the first display of the proof of Lemma~\ref{lem:h-to-hbar} for details.

The next lemma states that the simplified front is close to the 
process $X_L$ from \eqref{eq-XLdef} in sup-norm.
\begin{lemma} \label{lem:hbar-convergence}
We have
    \begin{align}
        \lim_{L\to\infty} \sup_{s\in[0,T], \theta \in \S^{d-2}}| 8^{-\frac{1}4}\hs(s,\theta)- X_L(s)| = 0 \text{ in $\Pb$-probability.}
    \end{align}
\end{lemma}

\begin{proof}[Proof of \cref{prop:sup-norm}]
Fix any $\delta >0$. The coupling of $\h$ and $\hs$ from \cref{lem:h-to-hbar} on $(\mathsf{\Omega}, \mathsf{F}, \mathsf{P}^{(b)})$ states that, for large enough $L$ and with probability $1-o(1)$,
\[
\sup_{s\in[0,T],\theta\in\S^{d-2}} |L^{-3/2}\hcirc(sL,\theta) - \hsb(s,\theta)| <\delta/2 
\]
\cref{lem:hbar-convergence} allows us to construct $X_L$ on $(\mathsf{\Omega}, \mathsf{F}, \mathsf{P}^{(b)})$ such that $|8^{-1/4}\hsb(s,\theta)-X_L(s)| < \delta/2$ for all $s\in [0,T]$ and $\theta \in \S^{d-2}$ and large enough $L$,  with probability $1-o(1)$. The proposition follows.
\end{proof}

\cref{lem:h-to-hbar,lem:hbar-convergence} are proved in \cref{sec:simplifications,sec:proof-hbar-convergence} respectively. 
An  outline of the proof of each result is provided at the start of each of these sections.

\section{Estimates on the trajectory of the spine and the proof of \cref{prop:X_L-rho-convergence}}
\label{sec:spine}
In this section, we provide the estimates needed to control the spinal trajectory $\cS_.(\xi)$, defined in \cref{subsec:extremallimitcluster}, throughout the rest of the article.
The section ends with a proof of \cref{prop:X_L-rho-convergence}. 
Recall the asymptotic notation from \cref{subsec:asymp-notation} as well as $\Pb$ from \cref{def:Pb}.

We begin with a preliminary input from \cite{Mallein15}, who proved a right-tail bound on the maximum displacement of BBM in any dimension $D$.
\begin{lemma}[{\cite[Eq.(1.2)]{Mallein15}}]\label{lem:mallein}
Fix $D\geq 1$, and let $R_s^*$ denote the maximum norm at time $t$ of a  standard $D$-dimensional BBM. There exists  $C>0$ such that for all $s\geq 1$ and $z\in [1,s^{1/2}]$,
\begin{align}
    C^{-1}ze^{-\sqrt2z} \leq \P\big(R_s^* \geq \sqrt2s + \tfrac{D-4}{2\sqrt2}\log s+z \big) \leq C ze^{-\sqrt2z}\,. \label{eqn:mallein}
\end{align}
\end{lemma}

The next estimate gives upper and lower barriers on the size of the components of $\cS_s(\xi)$, for all $s$ large enough. For any $\nu >0$, define the event 
\begin{align}
    E_{\xi}^{\Bumpeq}(\nu,L) := \{\forall s \geq L :  \|\Y_s(\xi)\|/s^{\frac{1}2} \in [s^{-\nu}, s^{\nu}], ~\bA_s(\xi)/s^{\frac{1}2} \in [s^{-\nu}, s^{\nu}]\}\,. \label{def:spine-barrier-event}
\end{align}

\begin{lemma}[Barrier on the spinal trajectory] \label{lem:barrier-spine}
For any fixed $\nu >0$, we have
\begin{align*}
    \lim_{L\to \infty} \Pb\big( E_{\xi}^{\Bumpeq}(\nu,L)  \big) = 1\,.
\end{align*}
\begin{proof}
From \cref{eqn:conditioned-law-A}, the $\Pb$-law of $\hat{A}_.(\xi)$ is absolutely continuous with respect to the law of $-\Gamma_.^{(b)}$, which is given in \eqref{def:Gamma}. The law of iterated logarithm for $-\Gamma^{(b)}$ (which holds because eventually $-\Gamma^{(b)}$ evolves as a Bessel(3) process), as an almost-sure statement, therefore applies to $\hat{A}_.(\xi)$. The result for $\|\Y_\cdot(\xi)\|$ follows from the  law of iterated logarithm for Bessel processes. 
\end{proof}
\end{lemma}

Recall the last exit time $\tau_C(\mathcal{X})$ of a stochastic process $\mathcal{X}_.$ from~\eqref{def:last-exit-time}. The next lemma shows that $L^{-2}\tau_{L}(\bA(\xi))$ is tight in $L$. 

\begin{lemma}[Last exit time estimate] \label{lem:last-exit-time} 
Define the event 
\begin{align}
    E_{\xi}^{\exit}(K,L) := \{\tau_L(\bA(\xi)) \in [K^{-1} L^2, KL^2] \} \,. \label{def:E-exit}
\end{align}
Then
\[
\lim_{K \to \infty} \inf_{L\geq 1} \Pb \Big( E_{\xi}^{\exit}(K,L)\Big) =1 \,.
\]
\begin{proof}
Bounding the exponential tilt in the right-hand side of \cref{eqn:conditioned-law-A} by $1$, we  have 
\begin{align*}
    \Pb(E_{\xi}^{\exit}(K,L)^c) \leq \frac{1}{Z(b)}\P\Big(\tau_L(-\Gamma^{(b)}) \not \in [K^{-1}L^2, KL^2] \Big)\,.
\end{align*}
Recall the Brownian motion $B_.$ and the hitting time $T_b^B$ from the definition of $\Gamma^{(b)}$ \eqref{def:Gamma}. Note the decomposition $\tau_L(-\Gamma^{(b)}) = T_b^B+ \tau_{L+b}(R)$. The result then follows from Brownian scaling of the Bessel process, which  in particular implies that $\tau_{L+b}(R)/(L+b)^2$ takes the same law as $\tau_1(R)$. 
\end{proof}
\end{lemma}

Next,  we describe the process $L^{-1}\bA_{\cdot L^2}(\xi)$ as an approximate Bessel(3) process.
\begin{lemma}[Bessel(3) convergence]\label{lem:spine-bessel-convergence}
For any fixed $N>0$, the $\Pb$-law of 
$
    (L^{-1}\bA_{\sigma L^2}(\xi))_{\sigma \in [0,N]} 
$
converges weakly in $(C[0,N],\|\cdot\|_{L^{\infty}})$ to the law of a Bessel(3) process as $L$ tends to $\infty$.
\end{lemma}
\begin{proof}
Let $f: (C[0,N], \|\cdot\|_{\infty}) \to \R$ be a bounded, continuous function, and let $\E^{(b)}$ denote expectation under $\Pb$. From \cref{eqn:conditioned-law-A}, we have
\begin{align} \label{eqn:app-mallein}
    \E^{(b)}\bigg[f\Big(\tfrac{\hat{A}_{\sigma L^2}(\xi)}{L}\Big)_{\sigma \in [0,N]} \bigg]
    = \frac{1}{Z(b)} \E\bigg[e^{-2\int_0^{\infty} G_r(\sqrt2 \Gamma_r^{(b)})\d r} f\Big(\tfrac{-\Gamma^{(b)}_{\sigma L^2}}{L}\Big)_{\sigma \in [0,N]}\bigg] \,.
\end{align}
Now, by \cref{lem:barrier-spine}, we have $-\Gamma_r^{(b)} \geq r^{1/2 -\nu}$ for all $r \geq K$ and for any $\nu \in (0,1/2)$, with probability tending to $1$ as $K$ tends to $\infty$. An application of \cref{eqn:mallein} then yields 
\[
    G_r(\sqrt 2 \Gamma_r^{(b)}) \leq \P(M_r \geq \sqrt2 r + r^{1/2 - \nu}) \leq e^{-r^{1/2 - \nu}}
\]
for all $r \geq K$ and $K$ sufficiently large (note  we absorbed the cumbersome $\log r$ terms by reducing the exponential pre-factor).
Then, using the boundedness of $f$ and \cref{lem:barrier-spine}, we have 
\begin{align} \label{eqn:spine-bessel-convergence-1}
    \E^{(b)}\bigg[f\Big(\tfrac{\hat{A}_{\sigma L^2}(\xi)}{L}\Big)_{\sigma \in [0,N]} \bigg]
    = \frac{1+o_K(1)}{Z(b)} \E\bigg[e^{-2\int_0^{K} G_r(\sqrt2 \Gamma_r^{(b)})\d r} f\Big(\tfrac{-\Gamma^{(b)}_{\sigma L^2}}{L}\Big)_{\sigma \in [0,N]}\bigg] +o_K(1)\,.
\end{align}
Defining $M_{\Gamma}^{(b)} := \max_{\sigma \in [0, L^{-2}T_b^B]} |\Gamma_{\sigma L^2}^{(b)}|$ and $M_{R}^{(b)} := \max_{\sigma \in [0, L^{-2}T_b^B]} |R_{\sigma L^2}|$, we have
\begin{align*}
    \sup_{\sigma \in[0,T]} \bigg|\frac{-\Gamma_{\sigma L^2}^{(b)}}{L} - \frac{R_{\sigma L^2}}{L} \bigg| 
    \leq 
    \max\bigg( \frac{|M_{\Gamma}^{(b)}+M_{R}^{(b)}|}{L}  
    ~,~ \sup_{\sigma \in [L^{-2}T_b^B,T]} 
    \frac{|R_{(\sigma - L^{-2}T_b^B)L^2}- R_{\sigma L^2}|}{L} \bigg)\,.
\end{align*}
From Brownian scaling, $(L^{-1}R_{\sigma L^2})_{\sigma \geq 0}$ is equal in distribution to $(R_{\sigma})_{\sigma\geq 0}$, and thus we have
almost-sure $(1/2-\delta)$-H\"older continuity with H\"older constant whose law is independent of $L$, for any $\delta \in (0,1/2)$. It follows that for any $\delta' \in (0,1)$ and all $L$ large, the above difference is bounded by $L^{-1+\delta'}$ with probability $1-o_L(1)$. Continuity of $f$ then yields
\begin{align} \label{eqn:bessel-convergence-condexp}
    \E^{(b)}\bigg[f\Big(\tfrac{\hat{A}_{\sigma L^2}(\xi)}{L}\Big)_{\sigma \in [0,N]} \bigg]
    = \frac{1+o_{K}(1)}{Z(b)} \E\bigg[e^{-2\int_0^{K} G_r(\sqrt2 \Gamma_r^{(b)})\d r} f\Big(\tfrac{R_{\sigma L^2}}{L}\Big)_{\sigma \in [0,N]}\bigg] +o_{L,K}(1) \,,
\end{align}
where the $o_{L,K}(1)$ term denotes a function that vanishes as first $L$ tends to $\infty$, then $K$ tends to $\infty$.

Observe that
\[
    \E\bigg[ f\Big(\tfrac{R_{\sigma L^2}}{L}\Big)_{\sigma \in [0,N]}  \given \Big(\tfrac{R_{\sigma L^2}}{L}\Big)_{\sigma \in [0,KL^{-2}]}\bigg] - \E\bigg[ f(R_{\sigma})_{\sigma \in [0,N]} \bigg]
\]
tends to $0$ in probability as $L$ tends to $\infty$. 
A quick way to see this is as follows: one may couple $(L^{-1} R_{\sigma L^2})_{\sigma \geq 0}$ with a Bessel(3) process $\bar{R}_{\sigma}$  in such a way that 
$
    L^{-1} R_{\sigma L^2} \in [\hat{R}_{\sigma}, \hat{R}_{\sigma} + M]
$
for all $\sigma \in [0,N]$, where  $M:= \sup_{\sigma \in [0,KL^{-2}]} L^{-1} R_{\sigma L^2}$ and
$\hat R_{\sigma} := 0$ for $\sigma \leq KL^{-2}$ and $\hat R_{\sigma} :=\bar{R}_{\sigma}$  for $\sigma> KL^{-2}$.
Brownian scaling implies  $M$ converges to $0$ a.s., and the observation follows.

Thus, taking first $L\to\infty$ then $K\to\infty$ in \cref{eqn:bessel-convergence-condexp}, we obtain $\E[f(R_{\sigma})_{\sigma \in [0,N]}]$ in the limit. 
\end{proof}

Finally, we prove \cref{prop:X_L-rho-convergence}.
\begin{proof}[Proof of \cref{prop:X_L-rho-convergence}]
Observe that $X_L(t)$ and $\rho(t)$ are a.s.\ positive (take $\sigma = 0$ in the supremums). Thus, it will suffice to show $X_L^2$ converges weakly to $\rho^2$. For  $N \geq 0$, define the following:
\[
    X_{L,N}^2(t) := \sup_{\sigma \in [0, N]} \Big( \sigma t- \sigma \frac{\bA_{\sigma L^2}(\xi)}{L}\Big) \quad \text{ and } \quad 
    \rho_N^2(t) := \sup_{\sigma \in [0,N]} \Big(\sigma t - \sigma R_{\sigma}\Big)\,.
\]
Below, convergence of processes takes place in $(C[0,T], \|\cdot\|_{\infty})$.

\begin{claim}\label{claim:X_L-X_LN}
The process $(X_{L,N}^2(t))_{t\in [0,T]}$ converges to $(X_{L}^2(t))_{t\in [0,T]}$ in probability as $N$ tends to $\infty$, uniformly over $L\geq 1$. 
\begin{proof}[Proof of \cref{claim:X_L-X_LN}]
We seek to show 
\[
\lim_{N\to\infty} \sup_{L\geq 1} \P\big( \|X_{L,N}^2 - X_L^2\|_{L^{\infty}[0,T]} >0 \big) = 0\,.
\]
Recall the last exit time \eqref{def:last-exit-time}, and define $\mathbf{\hat{A}}^{(L)}_\sigma(\xi) := L^{-1}\bA_{\sigma L^2}(\xi)$.
Since $X_L^2 > 0$ a.s., we have 
\[
\Big\{ \|X_{L,N}^2 - X_L^2\|_{L^{\infty}[0,T]} = 0\Big\} \subset 
\Big\{ \sup_{t\in[0,T]} \sup_{\sigma > N} \sigma \big(t -  \mathbf{\hat{A}}^{(L)}_\sigma(\xi) \big)\leq0 \Big\} \subset \{ \tau_T(\mathbf{\hat{A}}^{(L)}_.(\xi)) \leq N\}\,.
\]
The last containment is a consequence of the simple fact that for all $\sigma > \tau_T(\mathbf{\hat{A}}^{(L)}_.(\xi))$ and for all $t\leq T$, we have $t- \mathbf{\hat{A}}^{(L)}_\sigma(\xi) <t -T <0$. 
The result now follows from \cref{lem:last-exit-time}.
\end{proof}
\end{claim}

\begin{claim}\label{claim:rho-rho_N}
    The process $(\rho_{N}^2(t))_{t\in[0,T]}$ converges to $(\rho^2(t))_{t\in [0,T]}$ in probability as $N$ tends to $\infty$.
\begin{proof}[Proof of \cref{claim:rho-rho_N}]
Just as in the proof of \cref{claim:X_L-X_LN}, the result follows from the convergence of $\P(\tau_T(R) > N)$ to $0$ as $N$ tends to $\infty$. 
\end{proof}
\end{claim}

\begin{claim}\label{claim:X_LN-rho_N}
For  $N> 0$, the process $(X_{L,N}^2(t))_{t\in [0,T]}$ converges weakly to $(\rho_{N}^2(t))_{t\in [0,T]}$ as $L\to\infty$.
\begin{proof}[Proof of \cref{claim:X_LN-rho_N}]
Let $\cA_{L,N}$ denote the process $(L^{-1} \hat{A}_{\sigma L^2}(\xi))_{\sigma \in [0,N]}$, which  is an a.s.\ element of $C[0,N]$. Similarly, let $\cR_N$ denote $(R_{\sigma})_{\sigma \in [0,N]}$.
Consider the function
\begin{align*}
    g: C[0,N] &\to C[0,T] \\
    \phi &\mapsto g(\phi)(t) := \sup_{\sigma \in [0,N]} (\sigma t - \sigma \phi(\sigma))\,.
\end{align*}
\cref{claim:X_LN-rho_N} is equivalent to the following: for any bounded continuous $f: (C[0,T],\|\cdot\|_{\infty}) \to \R$,
\[
    \E [fg(\cA_{L,N})] \xrightarrow[L\to\infty]{} \E [fg(\cR_N)] \,.
\]
In light of the weak convergence of $\cA_{L,N}$ to $\cR_N$ provided by \cref{lem:spine-bessel-convergence}, we only need to show $g$ is continuous with respect to the uniform topologies on  $C[0,N]$ and $C[0,T]$. But this is an immediate consequence of the fact that, for any $\phi$ and $\tilde{\phi}$ in $C[0,N]$, one has 
\[
     \sigma t - \sigma \phi(\sigma) -  \|\phi- \tilde{\phi}\|_{L^{\infty}[0,N]} \leq  \sigma t - \sigma \tilde{\phi}(\sigma) \leq \sigma t - \sigma \phi(\sigma)+   \|\phi- \tilde{\phi}\|_{L^{\infty}[0,N]} \,. \qedhere
\]
\end{proof}
\end{claim}

We now complete the proof of \cref{prop:X_L-rho-convergence}. Take  $f: (C[0,T],\|\cdot\|_{\infty}) \to \R$ bounded and continuous. We have
\begin{align*}
    | \E[f(X_L^2) - \E[f(\rho^2)] | 
    \leq
    |\E[f(X_L^2) - f(X_{L,N}^2)]| + | \E[f(X_{L,N}^2) - f(\rho_N^2)]| + |\E[f(\rho_N^2) -f(\rho^2)]|\,,
\end{align*}
where we implicitly identify $X_L^2$, $\rho^2$, $X_{L,N}^2$, and $\rho_N^2$ with their restrictions to $[0,T]$. The first term tends to $0$ as $N\to\infty$, uniformly in $L$, by \cref{claim:X_L-X_LN}. The second term tends to $0$ as $L\to\infty$ by \cref{claim:X_LN-rho_N}. The third term tends to $0$ as $N\to\infty$ by \cref{claim:rho-rho_N}. Thus, sending first $L$ to $\infty$ then $N$ to $\infty$ yields the proposition.
\end{proof}

We end with a comment on the H\"older continuity of $L^{-1}\bA_{\cdot L^2}(\xi)$. Fix  $\delta \in (0,1/2)$. Observe that $-\Gamma^{(b)}_.$ is almost-surely $(1/2-\delta)$-H\"older continuous, being the concatenation of two almost-surely $(1/2-\delta)$-H\"older continuous functions. Due to Brownian scaling, $L^{-1}\Gamma^{(b)}_{\cdot L^2}$ is also almost-surely $(1/2-\delta)$-H\"older with a H\"older constant whose law is independent of $L$. The same is true for $L^{-1}\bA_{\cdot L^2}(\xi)$ by absolute continuity.
Letting $\fC_L$ denote the $(1/2-\delta)$-H\"older constant of $L^{-1}\bA_{\cdot L^2}(\xi)$, we record
\begin{align}
    \lim_{K\to\infty} \sup_{L\geq1}\Pb( \fC_L \geq K) = 0\,. \label{eqn:bound-A-holder-constant}
\end{align}

\section{Proof of \cref{lem:h-to-hbar}}
\label{sec:simplifications}
The goal of this section is to prove \cref{lem:h-to-hbar} via a series of reductions on $\h$, outlined as follows.
Recall \eqref{eqn:extremal-front-double-maximum}, which gives $\h(sL,\theta)$ as the maximum of the $Z_L^{i*}(s,\theta)$ over $i \in \N$, where $Z_L^{i*}(s,\theta)$ denotes the contribution of the conditioned BBM point cloud born at time $\tau_i$ to $\h(sL,\theta)$.
Our first step,
\cref{lem:earlytimes}, shows via union bound that, for $i \in \N$ such that $\tau_i < L^{1.4}$, the contribution $Z_L^{i*}(s,\theta)$ is tiny compared to the $L^{-3/2}$ scaling.
The next step is \cref{lem:no-conditioned-clans}, which
states that conditioning on $E_i$ (defined in \eqref{def:Ei}), for all $i\in\N$ such that $\tau_i \geq L^{1.4}$, is asymptotically trivial. This is simply because $A_s(\xi)$ decays to $-\infty$ far faster than the maximum of a BBM grows (see \eqref{eqn:mallein} for tail decay), and so $\{\max_{v\in \cN_s} X_s(v) + A_s(\xi) <0\}$ occurs with probability exponentially close to $1$ in $s$.
In particular, we may replace the conditioned BBMs $\cN_{\tau_i}^{i*}$ with standard BBMs in the definition of $Z_L^{i*}$--- we call this quantity $\bar{Z}_L^i$, defined in \eqref{def:Z-bar}.
Working with standard BBMs will allow us to perform a first-moment upper bound to eliminate the contribution of all BBM point clouds born after the last exit time $\tau_{TL}(\bA(\xi))$ of the interval $[0,TL]$ by $\bA_.(\xi)$--- this is \cref{lem:times-larger-L2}. The proof of \cref{lem:h-to-hbar} follows, and is given at the end of the section.

\subsection{Preliminaries}
Here, we record two inputs from the theory of BBM. 
The first is the
 many-to-one lemma, 
 a standard tool in the study of spatial branching processes. 
 See \cite{harris-roberts} for a thorough discussion of many-to-few lemmas. In this simple case, the result follows from the indepedence of the branching times from the particle trajectories and the tower property of expectation.
\begin{lemma}[Many-to-one lemma]\label{lem:many-to-one}
Fix $D\geq 1$, and let $\{\B_s(v):v \in \cN_s\}_{s\geq 0}$ denote a standard $D$-dimensional BBM. Let $\B_.$ denote a standard $D$-dimensional Brownian motion. For any $S\geq 0$ and any measurable function $f: C[0,S] \to\R$, we have 
\[
    \E\Big[ \sum_{v\in \cN_S}f((\B_s(v))_{s\in[0,S]}) \Big] = e^S \E[f((\B_s)_{s\in[0,S]})]\,.
\]
\end{lemma}

In \cref{lem:mallein}, we quoted a right-tail bound on the maximum displacement of BBM, in a certain regime of the tail. The following provides an (extremely) crude bound that holds in all regimes.
\begin{lemma}\label{lem:crude-right-tail}
Using the notation of \cref{lem:mallein}, there exists  $C>0$ such that for all $s,z>0$,
\begin{align}
    \P(R_s^* \geq z) \leq C e^{s - \frac{z^2}{3s}}\,. 
\label{eqn:crude-right-tail}
\end{align}
\begin{proof}
A union bound and the many-to-one lemma yields 
\begin{align*}
    \P(R_s^* \geq z) = \P( \cup_{v\in \cN_s} \{ \|\B_s(v)\| \geq z\}) \leq e^s \P(\cR_s \geq z)\,,
\end{align*}
where $\cR_.$ denotes a Bessel(D) process. For fixed $s>0$, $s^{-1/2}\cR_s$ is a chi random variable with $D$ degrees of freedom, so that 
\begin{align}
    \P(\cR_s \geq z) = C_0^{-1}\int_{z/\sqrt{s}}^{\infty} x^{D-1} e^{-x^2/2} \d x \leq C e^{- \frac{z^2}{3s}}\,,
    \label{eqn:bessel-onepoint-bound}
\end{align}
where $C_0$ denotes a positive constant. The result follows from the last two displays.
\end{proof}
\end{lemma}

\subsection{Simplification of the  front}
\label{subsec:simplifications}
Recall from the first line of \cref{subsec:proof-thm-extremalfrontC} that we have fixed $\ep \in (0,1)$ and a time-horizon $T>0$, and also that we condition on $\{A^*(\xi) \in \d b\}$. 
Recall the conditioned measure $\Pb$ from \cref{def:Pb}.

Our first estimate eliminates the contribution of BBM clouds born before time $L^{1.4}$.

\begin{lemma}\label{lem:earlytimes}
We have
\begin{align*}
    \lim_{L\to\infty} \Pb\Big(L^{-3/2}\max_{i \in \N} \max_{s \geq 0 \,,\, \theta\in \S^{d-2}} Z_L^{i*}(s,\theta) \ind{\tau_i < L^{1.4}} <  2L^{-0.1} \Big)  = 1\,.
\end{align*}

\begin{proof}
Dropping both indicator functions from the definition of $Z_L^{i*}(s,\theta)$ yields a quantity with no dependence on $s$, $\theta$, nor the $\{X_{\tau_i}^*(v): v\in \cN_{\tau_i}^{i*}\}$ for any $i \in \N$. Thus, 
it suffices to show
\begin{align*}
    \lim_{L\to\infty} \Pb\Big(\max_{i \in \N} \max_{v \in \cN_{\tau_i}^{i*}} \|\Y_{\tau_i}(v) + \Y_{\tau_i}(\xi)\| \ind{\tau_i < L^{1.4}} \geq   2L^{1.4} \Big)  = 0\,.
\end{align*}
Since $\sup_{s \in [0, L^{1.4}]} \|\Y_{s}(\xi)\|$ is bounded by $L^{1.3}$ with high probability (any exponent above $0.7$ works, by \cref{lem:barrier-spine}), we may ignore the $\Y_.(\xi)$ term in the above display. One is left with bounding the maximum norm of  $\mathbf{N}$ independent $(d-1)$-dimensional BBM clouds, where 
$\mathbf{N} = |\{i \in \N : \tau_i < L^{1.4}\}|$.
Let $R_s^*$ have the law of the maximum norm of a $(d-1)$-dimensional BBM  at time $s$. Then using the fact that $\mathbf{N} \leq L^2$ holds  with high probability, we find
\begin{multline*}
    \Pb\Big(\max_{i \in \N} \max_{v \in \cN_{\tau_i}^{i}} \|\Y_{\tau_i}(v)\| \ind{\tau_i < L^{1.4}} \geq   2L^{1.4} \Big)
    \leq
        \E\bigg[\sum_{i=1}^{\mathbf{N}} \P(R_{\tau_i}^* \geq L^{1.4}\given \tau)\ind{\mathbf{N} \leq L^2} \bigg]  +o(1) \\
    \leq L^2 \P(R_L^* \geq 2L^{1.4}) +o(1)\,.
\end{multline*} 
The right-tail bound \eqref{eqn:crude-right-tail} bounds the last line as $o(1)$, completing the proof.
\end{proof}
\end{lemma}

The next result states that the conditioning on the BBM clouds born after time $L^{1.4}$ is asymptotically trivial.
Recall the independent BBM $\{\B_s(v) = (X_s(v), \Y_s(v)) : v\in \cN_s^i\}_{s\geq0}$, as defined above \cref{def:ZLi}.

\begin{lemma}\label{lem:no-conditioned-clans}
For each $i\in \N$, define the event 
    $\bar{E}_i := \{\max_{v\in\cN_{\tau_i}^{i}} X_{\tau_i}(v)+A_{\tau_i}(\xi) <0\}$. Then
\begin{align*}
    \lim_{L\to\infty} \Pb\Big( \bigcap_{i \in \N} \bar{E}_i \cap \{\tau_i \geq L^{1.4}\}  \Big) = 1\,.
\end{align*}
\begin{proof}
Fix any $\nu \in (0,1/2)$. Using a union bound, as well as the high-probability lower bound $\bA_s \geq s^{1/2- \nu}$ for all $s\geq L$ provided by \cref{lem:barrier-spine}, we find
\begin{align*}
    \Pb \Big(\bigcup_{i\in \N} \bar{E}_i^c \cap \{\tau_i \geq L^{1.4}\} \Big)
    \leq \E\Big[ \sum_{i\in \N} \Pb\Big(\max_{v \in \cN_{\tau_i}^{i}} X_{\tau_i}(v) > \sqrt2 \tau_i+ \tau_i^{1/2-\nu} \given \tau_i\Big) \ind{\tau_i \geq L^{1.4}}\Big] +o(1)\,.
\end{align*}
The above $\Pb$-term may be bounded using \cref{eqn:mallein}. This yields an upper bound of
\[
    C \, \E\Big[\sum_{i \in \N} \exp(-\tau_i^{1/2-\nu}) \ind{\tau_i\geq L^{1.4}} \Big] +o(1)\,,
\]
where, as in \eqref{eqn:app-mallein}, we have absorbed the pre-factor and log terms by reducing the exponent of $\tau_i$.
The proof is finished by the Campbell formula and the upper bound on the intensity of $\pi$ by $2$.
\end{proof}
\end{lemma}

The next result shows that, 
for all  $i \in \N$  such that $\tau_i >\tau_{TL}(\bA(\xi))$, the quantity
\begin{equation} \label{def:Z-bar}
\bZ_L^{i}(s,\theta) := \max_{v\in \cN_{\tau_i}^{i}} \| \Y_{\tau_i}(v) + \Y_{\tau_i}(\xi)\| \ind{X_{\tau_i}(v) + A_{\tau_i}(\xi) \in (-sL, -sL+1]} \ind{|\theta \cdot \arg (\Y_{\tau_i}(v)+\Y_{\tau_i}(\xi))| \geq 1-\ep}  
\end{equation}
is negligible compared to $L^{3/2}$, where we recall the last exit time $\tau_C(\mathcal X)$ from \eqref{def:last-exit-time}.

\begin{lemma} 
\label{lem:times-larger-L2}
 For any $\nu <1/12$, 
\[
    \Pb \Big( \max_{i \in \N} \max_{s\in [0,T], \theta \in \S^{d-2}}  \bZ_L^{i}(s,\theta) \ind{\tau_i > \tau_{TL}(\bA(\xi))} \leq L^{5/4}  \Big) = 1+o(1)\,.
\]

\begin{proof}
We begin with some general calculations. 
Note that 
\[
    \max_{s \in [0,T],\theta \in \S^{d-2}} \bZ_{L}^{i}(s,\theta) =  \bZ_L^{i}(T)\,, \text{ where } \bZ_L^{i}(T):= \max_{v \in \cN_{\tau_i}^{i}} \| \Y_{\tau_i}(v) + \Y_{\tau_i}(\xi) \| \ind{X_{\tau_i}(v) + A_{\tau_i}(\xi) \in (-TL,0]}\,.
\]
Fix $z>0$ and a (possibly infinite) interval $I \subset [TL, \infty]$.  
Define an independent $d$-dimensional BBM $\{(X_s(v), \Y_s(v)) \in \R\times \R^{d-1} : v \in \cN_s\}_{s\geq 0}$, and define 
\[
    \bZ_L(s, T) := \max_{v \in \cN_s} \| \Y_s(v) + \Y_s(\xi) \| \ind{X_s(v) + A_s(\xi) \in (-TL, -0]}\,.
\]
Recall the barrier event $\Ebump(\nu,L)$ from \eqref{def:spine-barrier-event}. Using the Campbell formula and the bound on the intensity of $\pi$ by $2$, we compute
\begin{multline}
    \Pb\big( \exists i\in \N : \bZ_L^{i}(T)\ind{\tau_i \in I} > z \,,\, \Ebump (\nu, L)\big) 
    \leq  \E^{(b)} \Big[\indset{\Ebump (\nu, L)} \E^{(b)}\Big[ \sum_{i \in \N} \ind{\bZ_L^{i}(T)>z} \ind{\tau_i \in I} \given \xi \Big]\Big] \\
    \leq 2\,\E^{(b)} \Big[\indset{\Ebump (\nu, L)}  \int_I \P^{(b)}\big( \bZ_L(s, T) >z \given \xi\big) \d s  \Big]\,.    \label{eqn:general-formula}
\end{multline}
Applying a union bound and the many-to-one lemma (\cref{lem:many-to-one}), we find
\begin{multline}
    \Pb\big( \bZ_L(s,T) >z \given \xi\big) \leq 
    \E^{(b)} \Big[ \sum_{v\in \cN_s} \ind{X_s(v)+A_s(\xi) \in (-TL,0]} \ind{\|\Y_s(v)+ \Y_s(\xi)\|>z} \given \xi \Big] \\
    = e^s \, \Pb\big(X_s(v) \in [\sqrt2 s- (TL- \bA_s(\xi)), \sqrt{2}s+\bA_s(\xi)] \given \xi \big) ~\Pb\big(\|\Y_s(v) + \Y_{s}(\xi)\| >z \given \xi \big)\,,
    \label{eqn:bound-Z_L}
\end{multline}
recalling $\bA_s(\xi) := -A_s(\xi)-\sqrt2s$.
Since $X_s(v) \sim \cN(0,s)$, a Gaussian tail bound yields
\begin{multline}
    e^s \Pb(X_s(v) \in [\sqrt2 s- (TL- \bA_s(\xi)), \sqrt{2}s+\bA_s(\xi)] \given \xi ) \\
    \leq \ TL \, \exp \Big( -\frac{(TL-\bA_s(\xi))^2}{2s} + \sqrt2(TL- \bA_s(\xi)) \Big) \ \leq \ TL \, \exp \Big( \sqrt2(TL- \bA_s(\xi)) \Big)\,. \label{eqn:bound1-X}
\end{multline}

We now eliminate the contribution of those $\tau_i$ larger than $L^{2+6\nu}$, by showing that for such $\tau_i$,  no particle $v\in \cN_{\tau_i}^{i}$ satisfies $X_s(v) \in [\sqrt2 s- (TL- \bA_s(\xi)), \sqrt{2}s+\bA_s]$, with high probability. Indeed, bounding the second $\Pb$ term in~\eqref{eqn:bound-Z_L} by $1$, substituting this bound into~\eqref{eqn:general-formula} with $I := (L^{2+6\nu}, \infty)$, and then applying the bound  \eqref{eqn:bound1-X}, we find that for any $z>0$,
\begin{multline*}
    \Pb\big( \exists i\in \N : \bZ_L^{i}(T)\ind{\tau_i >L^{2+6\nu}]} > z \,,\, \Ebump (\nu, L)\big)
    \ \leq \ 2TL \, \E^{(b)} \Big[\indset{\Ebump (\nu, L)}  \int_{L^{2+6\nu}}^{\infty} e^{\sqrt2(TL- \bA_s(\xi))} \d s  \Big] \\
    \leq \ 2TL\int_{L^{2+6\nu}}^{\infty} e^{\sqrt2(TL- s^{\frac12-\nu})} \d s \ = \ o(1)\,.
\end{multline*} 
In the last step, we have used the definition of $E_{\xi}^{\Bumpeq}$ as well as $(2+6\nu)(1/2-\nu) > 1$ for $\nu <1/6$.

Next,  consider $\tau_i \in I:= [\tau_{TL}(\bA(\xi)),L^{2+6\nu}]$.
Due to \cref{lem:last-exit-time}, with probability $1-o(1)$, 
we may assume $\tau_{TL}(\bA(\xi))\in [L, L^{2+6\nu}]$ (the exact 
lower bound on the interval  is of no special importance, anything $o(L^2)$
would work).
Observe that on $\Ebump (\nu,L)$, 
\[
\max_{s \in [L, L^{2+6\nu}]} \|\Y_s(\xi)\| \leq L^{(2+6\nu)(1/2+\nu)} = o(L^{5/4})\,.
\]
In particular, on $\Ebump(\nu, L)$, we have the following bound for all $s \in I$:
\[
    \Pb\big( \|\Y_s(v) + \Y_s(\xi)\| >L^{\frac{5}4} \given \xi) \leq \Pb \big( \|\Y_s(v)\| >\tfrac{1}2L^{\frac{5}4} \given \xi) 
    \leq Ce^{-\frac{1}{12}L^{1/2-6\nu}}\,.
\]
In the second inequality, we used~\eqref{eqn:bessel-onepoint-bound} and $s\leq L^{2+6\nu}$. Observe that $1/2 -6\nu>0$ since $\nu <1/12$. 
Furthermore, for $s \geq \tau_{TL}(\bA(\xi))$, the bound in~\eqref{eqn:bound1-X} simplifies to 
\[
    e^s \Pb(X_s(v) \in [\sqrt2 s- (TL- \bA_s(\xi)), \sqrt{2}s+\bA_s] \given \xi) \leq TL \,.
\]
Substituting the last two bounds into~\eqref{eqn:bound-Z_L} then subsituting that bound into~\eqref{eqn:general-formula} yields
\begin{align*}
    \Pb\Big( \exists i\in \N : \bZ_L^{i}(T)\ind{\tau_i \in [\tau_{TL}(\bA(\xi)), L^{2+6\nu}]} > L^{5/4} \,,\, \Ebump (\nu,L)\Big) 
    \leq 2CT L^{3+6\nu}\exp\big(-\tfrac{1}{12}L^{\frac12-6\nu}\big)  = o(1)\,.
\end{align*}
This concludes the proof.
\end{proof}
\end{lemma}

We are now ready to prove \cref{lem:h-to-hbar}.
\begin{proof}[Proof of \cref{lem:h-to-hbar}]
  For  $s\geq 0$ and $\theta\in\S^{d-2}$, define $\mathsf{T}\h(sL,\theta) := \max_{i\in\N} Z_L^{i*}(s,\theta) \ind{\tau_i\geq L^{1.4}}$ 
and $\tilde{h}^{\ep}(s,\theta):=  L^{-3/2}\max_{i\in\N} \bZ_L^{i}(s,\theta) \ind{\tau_i\geq 1.4}$ 
as processes on $[0,T]\times\S^{d-2}$. 
\cref{lem:earlytimes} implies 
$\max_{s\in [0,T], \theta \in \S^{d-2}} \|L^{-3/2}\h(sL,\theta)- L^{-3/2}\mathsf{T}\h(sL,\theta)\|$ converges to $0$ in probability as $L\to\infty$, while
\cref{lem:no-conditioned-clans} implies  the total variation distance between $L^{-3/2}\mathsf{T}\h$ and $\tilde{h}^{\ep}$ 
converges to $0$ as $L\to \infty$. 
Thus, there exists a coupling between $(L^{-3/2}\h(sL,\theta))_{s\in[0,T],\theta\in\S^{d-2}}$ and $(\tilde{h}^{\ep}(s,\theta))_{s\in[0,T],\theta\in\S^{d-2}}$ such that the $L^{\infty}([0,T]\times \S^{d-2})$-norm of their difference converges to $0$ in probability.

Now, the discrepancy between $\bZ_L^{i}$ and $Z_L^{i}$ (defined above \cref{lem:h-to-hbar}) comes from the presence of the $\Y_{\tau_i}(\xi)$ term in $\bZ_L^{i}$. This is dealt with as follows.
Fix any   $\nu <1/12$.
On the event $E_{\xi}^{\exit}(K,TL) \cap \Ebump(\nu, L)$ (defined in \cref{sec:spine}), we have the following for  $\tau_i \leq \tau_{TL}(\bA(\xi)) \leq KT^2L^2$:
\[
    L^{-\frac32} \Big| \norm{\Y_{\tau_i}(v)+ \Y_{\tau_i}(\xi)} - \norm{\Y_{\tau_i}(v)} \Big|\leq L^{-\frac32} \norm{\Y_{\tau_i}(\xi)} \leq L^{-\frac32} \tau_i^{\frac12+\nu} \leq (KT^2)^{\frac12+\nu} L^{-\frac12+2\nu} = o_L(1)\,.
\]
This in conjunction with Lemmas~\ref{lem:barrier-spine} and~\ref{lem:times-larger-L2} yields the following convergence:
\begin{multline*}
    \lim_{L\to\infty} \Pb\Big(  \max_{s\in[0,T], \theta \in \S^{d-2}} \Big| \tilde{h}^{\ep}(s,\theta)  - \hs(s,\theta)\Big| > 2L^{-\frac{1}4} \Big) \\
    \leq \lim_{K\to\infty} \limsup_{L\to\infty} 
    \Pb\Big( E_{\xi}^{\exit}(K,TL)\cap  \Ebump(\nu, L) \cap 
    \big\{ \max_{i \in \N,s \in [0,T], \theta \in \S^{d-2}} \bZ_L^{i}(s, \theta) \ind{\tau_i > \tau_{TL}(\bA(\xi))} > L^{\frac{5}4} \big\}\Big)
    =0\,.
\end{multline*}
This convergence in probabilty statement along with the aforementioned coupling with $L^{-3/2} \h$ finishes the proof.
\end{proof}

\section{Proof of \cref{lem:hbar-convergence}}
\label{sec:proof-hbar-convergence}
Recall $\ep \in (0,1)$, $T>0$, and the conditioned measure $\Pb$ from the start of \cref{subsec:simplifications}.
In this section, we prove \cref{lem:hbar-convergence}, which states that the sup-norm of the difference between $\hs(s,\theta)$ and $X_L(s)$ converges to $0$ in probability, as $L\to\infty$.
Before proceeding, we give a sketch of the proof.

The following simple calculation is key. 
Fix $s \in [0,T]$ and $\ep \in \S^{d-2}$, as well as a Poisson time $\tau_i$ of order $L^2$. Suppose there exists a particle $v\in \cN_{\tau_i}^{i}$ satisfying:
\begin{align} \label{eqn:good-particle-conditions}
\begin{cases}
    X_{\tau_i}(v)+ A_{\tau_i}(\xi) \in (-sL, -sL+1]  \\
    \abs{\theta \cdot \arg\big(\Y_{\tau_i}(v) + \Y_{\tau_i}(\xi) \big)} \geq 1-\ep \\
    \|\B_{\tau_i}(v)\| = \sqrt2\tau_i + c(\tau_i) \text{ for some } |c(x)| = o(\sqrt{x})\,.
\end{cases}
\end{align}
In words, the above states that
 $X_{\tau_i}(v)$ and $\Y_{\tau_i}(v)$ lie in their ``target'' intervals (recall the definition of $Z_L^i$ from \eqref{def:ZLi}) while the norm $\|\B_{\tau_i}(v)\|$ lies sufficiently close to $\sqrt2 \tau_i$. The Pythagorean theorem yields
 \begin{align*}
    \|\Y_{\tau_i}(v)\| = (\|\B_{\tau_i}(v)\|^2 - | X_{\tau_i}(v)|^2)^{1/2} = 
    \sqrt{\big(2\sqrt2\tau_i + c(\tau_i) - (sL - \bA_{\tau_i}(\xi))\big)\big(c(\tau_i)  + (sL - \bA_{\tau_i}(\xi))\big)}\,.
\end{align*}
Now, suppose $sL - \bA_{\tau_i}(\xi)$ is of order $L$. This is reasonable because in the definition of $\hs$ given in \eqref{def:simplified-front}, we consider $\tau_i \leq \tau_{TL}(\hat{A}(\xi))$, and from \cref{lem:last-exit-time}, we know $\tau_{TL}(\bA(\xi)) = O(L^2)$. Moreover, $\bA_i(\xi)$ is an approximate Bessel process by \cref{lem:spine-bessel-convergence}, so $\bA_{\tau_i}(\xi)$ is typically of order $L$. The above then simplifies to
\[
    \|\Y_{\tau_i}(v)\| = 8^{\frac{1}4}L^{\frac{3}2} \Big(\frac{\tau_i}{L^2}\Big(s - \frac{\bA_{\tau_i}(\xi)}{L}\Big)\Big)^{\frac{1}2} + o(L^{\frac{3}2})\,.
\]
Writing $\tau_i = \sigma L^2$ and optimizing over $\sigma >0$ yields $X_L(s)$. 

An upper bound on $\hs(s,\theta)$ by $X_L(s)$ follows by observing that the maximum possible norm of any particle born at any time $\tau_i = O(L^2)$ is bounded by $\sqrt2\tau_i + (\log L)^2$; this is \cref{lem:total-upper-barrier}, and comes from a  union bound and the right-tail decay of the maximum norm of a BBM (\cref{eqn:mallein}). Now take $c(x) \equiv (\log L)^2 $ above.

For the lower bound, we need to show that for each $s\in[0,T]$ and $\ep \in \S^{d-2}$, there exists some $\tau_i$ of order $L^2$ and some particle $v\in \cN_{\tau_i}^{i}$ satisfying the conditions~\eqref{eqn:good-particle-conditions}. For this, we apply a classical result of G\"artner (\cref{prop:gartners-result} and \cref{cor:gartner}), which ensures this for $c(x) = -\frac{d+2}{2\sqrt2} \log x$.

As indicated in the above sketch, we prove each direction of the bound separately.
\begin{lemma}[Upper bound] \label{lem:hbar-convergence-upperbound}
We have the following for any $\eta >0$:
\begin{align}
  \lim_{L\to\infty} \Pb\Big(\max_{s\in[0,T], \theta \in \S^{d-2}} (8^{-\frac{1}{4}} \hs(s,\theta) - X_L(s)) > \eta \Big) = 0\,.
\end{align}
\end{lemma}

\begin{lemma}[Lower bound]\label{lem:hbar-convergence-lowerbound}
We have the following for any $\eta >0$,
\begin{align}
  \lim_{L\to\infty} \Pb\Big(\max_{s\in[0,T], \theta \in \S^{d-2}} (X_L(s) - 8^{-\frac{1}{4}} \hs(s,\theta)) > \eta \Big)  =0\,.
\end{align}
\end{lemma}

We will often work on the following high probability event for the spinal trajectory. 
For any $\nu \in (0,1/12)$, define
\begin{align}
    E^{*}_{\xi}(K,L) := E^{*}_{\xi}(K,L, \nu) =E_{\xi}^{\exit}(K,TL) \cap \Ebump(\nu, L) 
    \cap \{A^*(\xi) \in \d b\} \,,
    \label{def:E*}
\end{align}
where the events in the right-hand side were defined in 
\cref{lem:barrier-spine,lem:last-exit-time}. These lemmas show
\begin{align}
    \lim_{K\to\infty} \liminf_{L\to\infty} \Pb(E^*_{\xi}(K,L)) = 1 \,.
    \label{eqn:E*-probability-one}
\end{align}
The choice of $\nu$ here plays no role, and so we fix it arbitrarily for the remainder of the section.

\subsection{Proof of \cref{lem:hbar-convergence-upperbound}}
\label{subsec:pf-upper-bound}
The following lemma provides an upper-bound on the distance travelled by any particle born at time $\tau_i \leq L^3$ ($L^3$ is excessive, anything dominating $L^2$ would suffice).
\begin{lemma}\label{lem:total-upper-barrier}
Define the event 
\[
    E^{\max} := \{\forall i \in \N \,:\, \max_{v \in \cN_{\tau_i}^{i}} \|\B_{\tau}(v)\| \ind{\tau_i \in [0,L^3]} \leq \sqrt{2} \tau_i + (\log L)^2 \}\,.
\]
The probability of $E^{\max}$ tends to $1$ in the limit  $L \to \infty$.
\begin{proof}
This follows from the exponential right-tail  bound in \eqref{eqn:mallein} and 
a union bound over all Poisson points in $[0,L^3]$ (there are not more than, say, $L^4$   points in this interval, with high probability).
\end{proof}
\end{lemma}

\begin{proof}[Proof of \cref{lem:hbar-convergence-upperbound}]
Fix $\eta \in (0,1)$ and $K \in [1,L)$.
On the event $E_{\xi}^*(K,L) \cap E^{\max}$, the following calculation holds  uniformly over 
any $s\in [0,T]$, 
any Poisson point $\tau_i \in [L^{1.4}, \tau_{TL}(\bA(\xi))]$, and 
any $v \in \cN_{\tau_i}^{i}$ such that $X_{\tau_i}(v) + A_{\tau_i}(\xi) \in (-sL,-sL+1]$:
\begin{align*}
    \|\Y_{\tau_i}(v)\|^2  
    = \|\B_{\tau_i}(v)\|^2 - |X_{\tau_i}(v)|^2
    \leq \big(\sqrt2 \tau_i+ (\log L)^2\big)^2 - \big(\sqrt{2}\tau_i-(sL - \bA_{\tau_i}(\xi))\big)^2 \nonumber\\
    = 2\sqrt2\tau_i (sL - \bA_{\tau_i}(\xi) + (\log L)^2) - (sL - \bA_{\tau_i}(\xi))^2 + (\log L)^4 \,. \nonumber 
\end{align*}
Using $\tau_i \leq KL^2$ and  $\bA_{\tau_i}(\xi) \leq (KL^2)^{1/2+\nu}$ on $E_{\xi}^{\exit}(K,L) \cap E_{\xi}^{\Bumpeq}(\nu, L)$,
the previous display implies 
\begin{multline*}
    \|\Y_{\tau_i}(v)\|^2   \leq 2\sqrt2\tau_i (sL - \bA_{\tau_i}(\xi)) + CK^{\frac12+\nu}L^{2+4\nu} \\
     \leq 2\sqrt2 \sup_{\sigma \geq 0} \sigma L^2 (sL - \bA_{\sigma L^2}(\xi)) + CK^{\frac12+\nu}L^{2+4\nu} 
     = 2\sqrt2 L^3 X_L(s)^2  + CK^{\frac12+\nu}L^{2+4\nu} \,.
\end{multline*}
for some constant $C>0$ and $L$ large.
In light of the conditions imposed on $\tau_i$ and $v\in \cN_{\tau_i}^{i}$ at the start of the proof, the right-hand side above also bounds $L^3\hs(s,\theta)^2$. Thus, for $L$ sufficiently large,
\begin{align*}
    \Pb\Big(\max_{s\in[0,T], \theta \in \S^{d-2}} (8^{-\frac{1}{4}} \hs(s,\theta) - X_L(s)) > \eta \Big) \leq \Pb(E_{\xi}^*(K,L) \cap E^{max})\,.
\end{align*}
The lemma then follows from \cref{eqn:E*-probability-one} and \cref{lem:total-upper-barrier}.
\end{proof}

\subsection{Discretization and proof of \cref{lem:hbar-convergence-lowerbound}}
Let $\cM_L \subset [0,T]$ be a mesh of equally-spaced points such that $0\in \cM_L$ and the distance between $\cM_L$ and any point in $[0,T]$ is bounded by $1/2L$. For each $s\geq 0$, let $s_*:= s_*(L)$ denote the element of $\cM_L$ immediately to the left of $s$. Similarly, let $\cM_{\ep} \subset \S^{d-2}$ be a mesh of equally-spaced points such that the distance between $\cM_{\ep}$ and any point in $\S^{d-2}$ bounded by $\ep/2$. For each $\theta \in \S^{d-2}$, 
let $\theta_*$ denote an element of $\cM_\ep$ of minimal distance to $\theta$. Observe that, for any $s\in [0,T]$ and $\theta \in \S^{d-2}$, 
\[
    (-s_*L , -s_*L +1/2] \subset (-sL, -sL+1] ~\text{ and }~ 
    \{\psi \in \S^{d-2} : \psi \cdot \theta_* \geq 1-\tfrac{\ep}{2} \}  \subset \{\psi \in \S^{d-2} : \psi \cdot \theta \geq 1-\ep \} \,.
\]
The following quantity is exactly $Z_L^{i}$, defined in \cref{def:ZLi}, except with an interval of size $1/2$ instead of size $1$ in the first indicator and a cone of angle $\ep/2$ instead of $\ep$ in the second indicator:
\[
\hat{Z}^{i}(s,\theta) := \max_{v\in \cN_{\tau_i}^{i}} \|\Y_{\tau_i}(v)\| \ind{X_{\tau_i}(v) + A_{\tau_i}(\xi) \in (-sL, -sL+1/2]} \ind{|\theta\cdot \arg(\Y_{\tau_i}(v)+\Y_{\tau_i}(\xi))|\geq 1-\ep/2} \,.
\]
Therefore, we have $\hat{Z}^{i}(s_*,\theta_*)
\leq Z_{L}^{i}(s,\theta)$, and thus
\begin{align} \label{eqn:t*-lower-bound}
    \hat{h}^{\ep}(s_*,\theta_*):=  L^{-3/2} \max_{i\in\N} \hat{Z}^{i}(s_*,\theta_*) \ind{\tau_i \in[L^{1.4}, \tau_{TL}(\bA(\xi))]} \leq \hs(s,\theta)
\end{align}
for all $s\in [0,T]$ and $\theta \in \S^{d-2}$.
With this and \cref{claim:one-point-lowerbound} below, we can prove \cref{lem:hbar-convergence-lowerbound}. 

Before stating the claim, we define the event 
\[
    E_{\xi}^{**}(K,L) := E_{\xi}^*(K,L)  \cap \{\fC_L \leq K \}\,,
\]
where we recall that $\fC_L$ is the $(1/2-\delta)$-H\"older constant of $L^{-1}\bA_{\cdot L^2}(\xi)$, for some $\delta \in (0,1/2)$ (fix it arbitrarily). 
From \cref{eqn:bound-A-holder-constant,eqn:E*-probability-one}, 
\begin{align}
    \lim_{K\to\infty} \liminf_{L\to\infty} \Pb(E_{\xi}^{**}(K,L)) = 1\,.
\end{align}

\begin{claim}\label{claim:one-point-lowerbound}
Fix any $\eta, \iota>0$. There exists $C>0$ such that for all  $k\in [0,T]$ and $\psi \in \S^{d-2}$, 
\begin{align}
    \Pb\big( X_L(k) - 8^{-\frac{1}{4}}\hat{h}^{\ep}(k, \psi) > \eta \,,\, E_{\xi}^{**}(K,L) \big) \leq e^{-C L^{2-\iota}}\,.
\end{align}
\end{claim}
\cref{claim:one-point-lowerbound} is proved in \cref{subsec:pf-claim}, thereby completing the proof of \cref{thm:extremal-front}.

\begin{proof}[Proof of \cref{lem:hbar-convergence-lowerbound}]
Fix $\eta >0$. 
For any $s\geq 0$ and $\theta \in \S^{d-2}$, \cref{eqn:t*-lower-bound} gives us
\begin{align*}
    X_L(s) - 8^{-\frac{1}{4}}\hs(s,\theta)
    &\leq 
    X_L(s) - 8^{-\frac{1}{4}}\hat{h}^{\ep}(s_*,\theta_*) 
    \leq X_L(s_*) - 8^{-\frac{1}{4}}\hat{h}^{\ep}(s_*,\theta_*) 
    +|X_L(s) - X_L(s_*)|\,.
\end{align*}
Taking a supremum over $s \in [0,T]$ and $\theta \in \S^{d-2}$, we find
\begin{multline*}
    \sup_{s\in[0,T],\theta \in \S^{d-2}} \! X_L(s) - 8^{-\frac{1}{4}}\hs(s,\theta) 
    \leq \sup_{s\in [0,T]} |X_L(s) - X_L(s_*)| + 
    \sup_{k \in \cM_L, \psi \in \cM_{\ep}} \! X_L(k) - 8^{-\frac{1}{4}} \hat{h}^{\ep}(k,\psi).
\end{multline*}
We now employ a union bound:
\begin{multline*}
    \Pb\Big(\sup_{s\in[0,T],\theta \in \S^{d-2}} \! X_L(s) - 8^{-\frac{1}{4}}\hs(s,\theta)> \eta\Big)  
    \ \leq \ \Pb\Big(\sup_{s\in [0,T]} |X_L(s) - X_L(s_*)| > \eta/2\Big)  \\
    \ + \ \Pb(E_{\xi}^{**}(K,L)^c)
    \ + \sum_{k\in \cM_L, \psi \in \cM_\ep} \Pb\Big(X_L(k) - 8^{-\frac{1}4} \hat{h}^{\ep}(k,\psi) >\eta/2 \,, E_{\xi}^*(K,L)\Big)\,.
\end{multline*}
\cref{prop:X_L-rho-convergence} and the continuity of $\rho(\cdot)$ implies that the first term converges to $0$ as $L\to \infty$. The second term converges to $0$ by \cref{eqn:E*-probability-one}. The third term converges to $0$ by \cref{claim:one-point-lowerbound}. 
\end{proof}

\subsection{KPP Equation Input}
Consider the following F-KPP equation in $\R^{d}$:
\begin{equation}
\begin{cases}
    \partial_t u = \frac{1}2 \Delta u  + u(1-u) &\text{ in } \R^{d} \\
    u(0, x) = \indset{\norm{x} \leq 1/4} &\text{ for } x\in \R^{d}
\end{cases}\,.
\label{eqn:f-kpp}
\end{equation}
The seminal result of Skorohod \cite{Skorohod}, see also \cite{INW,mckean}, gives a representation for the (unique) solution of the above F-KPP equation in terms of  $d$-dimensional BBM. We give here McKean's version:
\begin{proposition}[McKean]\label{prop:mckean}
Equation~\eqref{eqn:f-kpp} has the unique solution
\[
    u(t,x) := \E_{x}\Big[1- \prod_{v \in \cN_t} (1-\indset{\norm{\B_t(v)}\leq 1/4}) \Big] 
    =  \P\Big(\exists v \in \cN_t: \norm{\B_t(v)-x}\leq 1/4 \Big) \,, \quad \forall x \in \Rd\,, 
    ~\forall t \geq 0\,.
\]
\end{proposition}
The following result comes from an important paper of G\"artner \cite{gartner}, who identified the location of the ``front'' for a more general family of PDEs, known as \emph{KPP equations}. For the PDE~\eqref{eqn:f-kpp}, the front is located at radial distance 
\begin{align}
    m_t^G(d) := \sqrt{2}t - \frac{d+2}{2\sqrt{2}}\log t \,. \label{def:gartner-shift}
\end{align} 
\begin{proposition}[G\"artner] \label{prop:gartners-result}
    For each $\fc\in (0,1/2)$, there exists some $r:= r_{\fc} >0$ such that for all $t\geq 1$,  and for all $x \in \R^d$ satisfying 
    \begin{equation}
\label{eq-range}
        \| x \| \in \big[m_t^G(d) - r,m_t^G(d)+r\big] \,,
    \end{equation}
    we have
    \begin{align}
        u(t,x) = \P\Big(\exists v \in \cN_t: \norm{\B_t(v)-x}\leq 1/4 \Big) \in [\fc,1-\fc]\,. \label{eqn:gartner}
    \end{align}
\end{proposition}
See  \cite{Roq19} for asymptotics of $u(t,x)$ for $x$ outside the range \eqref{eq-range}, which we do not need.
    \begin{proof}
        This follows from \cite[Theorem~4.1 and Remark~4.1]{gartner} applied with $f(u)=u(1-u)$ and  $g(x) = \bar{g}(\norm{x})=\indset{\norm{x}\leq 1/4}$ (so that  $v^*=\sqrt{2}$, and the assumptions in Equations~4.1,~4.18, and~4.19 of \cite{gartner} are easily seen to be satisfied).
    \end{proof}

We fix $\fc = 1/4$, and define $r:= r_{1/4}$ as in \cref{prop:gartners-result}. 
For any $\tau,s\geq 0$ and $\theta \in \S^{d-2}$,
let $\mathbf{z}_{\tau}(s,\theta) := (\mathbf{z}_{\tau}^{(1)}(s,\theta), \mathbf{z}_{\tau}^{(2)}(s,\theta)) \in \R \times \R^{d-1}$ denote the unique point in $\R^d$ satisfying 
\begin{align}
\begin{cases}
    \mathbf{z}^{(1)}_{\tau}(s,\theta) + A_{\tau}(\xi)=  - sL +1/4 \\
    \arg(\mathbf{z}^{(2)}_{\tau}(s,\theta)) = \theta 
    \\
    \norm{\mathbf{z}_{\tau}(s,\theta)} =m_{\tau}^G(d) - r_{1/4}
\end{cases}
\label{def:z}
\end{align} 
(compare with the conditions in~\eqref{eqn:good-particle-conditions}).
For $\mathbf{x}\in \R^d$ and $a>0$, let $B(\mathbf{x}, a)$ denote the ball in $\R^d$ of radius $a$ centered at $\mathbf{x}$.
\cref{cor:gartner} below states that, with very high probability, there exists some branching time $\tau_i \in \pi$ such that the associated BBM point cloud places a particle in $B(\mathbf{z}_{\tau_i}(s,\theta),\fc)$, for any $s \in[0,T]$ and  $\theta\in \S^{d-2}$. Moreover, we can localize this $\tau_i$ to lie in a specified interval.
\begin{corollary}\label{cor:gartner}
For any $s\in[0,T]$, $\theta \in \S^{d-2}$, and any nonempty interval $I \subset [1,\infty)$, we have
\[
    \Pb\Big(\exists \tau_i \in I \,,\, \exists v\in \cN_{\tau_i}^{i} : \mathbf{B}_{\tau_i}(v) \in B(\mathbf{z}_{\tau_i}(s,\theta),1/4) \Big) \geq 1- e^{-|I|/2} \,.
\]
\begin{proof}[Proof of \cref{cor:gartner}]
The probability in question may be expressed as
\begin{multline*}
    1- \Pb\Big(\forall \tau_i \in I \,,\, \forall v\in \cN_{\tau_i}^{i} \,:\, \norm{\mathbf{B}_{\tau_i}(v) - \mathbf{z}_{\tau_i}(s,\theta)} > 1/4 \Big) \\
    = 1- \E^{(b)}\bigg[ \prod_{\tau_i \in \pi}  \P\big(\forall v\in \cN_{\tau_i}^{i}:\norm{\mathbf{B}_{\tau_i}(v) - \mathbf{z}_{\tau_i}(s,\theta)} > 1/4 \given \tau_i \big) \ind{\tau_i \in I}\bigg] \,.
\end{multline*}
Since we consider $\tau_i \geq 1$, 
we may apply \cref{prop:gartners-result} to bound  the conditional probability term in the last line from above by $1-\fc =3/4$. The probability in question is then bounded  above by
\[
    1- \E^{(b)}[ (1-\fc)^{\#\{i \in \N \,:\, \tau_i \in I\}} ]\,.
\]
Let $\mathcal{P}$ denote a Poisson random variable with parameter $2|I|$. Using the fact that the intensity of $\pi$ is bounded by $2$, we find that the previous display is bounded by 
$
    1- \E^{(b)}[(1-\fc)^{\mathcal{P}}] = 1- e^{-2\fc|I|}\,,
$
where we used the formula for the moment generating function of a Poisson random variable.
\end{proof}
\end{corollary}

\subsection{Proof of \cref{claim:one-point-lowerbound}}
\label{subsec:pf-claim}
Fix $k \in [0,T]$ and $\psi \in \S^{d-2}$. 
Define 
\[
    \sigma_k^* := \mathrm{argmax}_{\sigma \geq 0}~\sigma \Big(k -  \frac{\bA_{\sigma L^2}(\xi)}{L}\Big)\,.
\]
 Note that  
$
k- L^{-1}\bA_{\sigma L^2}(\xi) < 0
$
for all $\sigma \geq \tau_{TL}(\bA(\xi))/L^2$. Thus, 
\[
    \sigma_k^* \in [0, \tau_{TL}(\bA(\xi))/L^2] ~\text{a.s.}
\] 
In particular, on $E_{\xi}^{\exit}(K,TL)$ (a subset of $E_{\xi}^{**}(K,L)$ and defined in \eqref{def:E-exit}), we have $\sigma_k^* \leq K$ a.s.
We may then assume 
\begin{align} \label{eqn:bound-A-from-k}
    L^{-1} A_{\sigma_k^*L^2} \leq k - \frac{\eta}{K}\,,
\end{align}
as otherwise, $X_L(k) <\eta$ on $\{\sigma_k^* \leq K\} \cap E_{\xi}^{**}(K,L)$, and so the claim becomes trivial.

Recall $\mathbf{z}_{\tau_i}(s,\theta)$ from \eqref{def:z}. For an interval $I \subset [1,\infty)$, define the event 
\begin{align} \label{def:percolation-event}
    E(k,\psi, I) := \{\exists \tau_i \in I \,,\, \exists v \in \cN_{\tau_i}^{i} : \mathbf{B}_{\tau_i}(v) \in B(\mathbf{z}_{\tau_i}(k,\psi),\tfrac14)\}\,.
\end{align}
\begin{claim}\label{claim:perc-lb}
Consider any $I\subset [1, \tau_{TL}(\bA(\xi))]$. On  
$E_{\xi}^{**}(K,L) \cap E(k,\psi, I)$, there exists a constant $C>0$ such that for all $L\geq K \geq 1$, 
we have 
\begin{align}
    \|\Y_{\tau_i}(v)\|^2 \geq 2\sqrt2\tau_i(kL- \bA_{\tau_i}(\xi))- CK^{\frac{1}2+\nu} L^{2+4\nu}\,, \text{ for $\tau_i$ and $v$ as in~\eqref{def:percolation-event}.}
\end{align}
\begin{proof}
From the Pythagorean theorem and the definitions of $\bA(\xi)$ \eqref{eqn:conditioned-law-A} and $m_t^G(d)$ \eqref{def:gartner-shift}, we have 
\begin{multline*}
    \|\Y_{\tau_i}(v)\|^2 \geq (m_{\tau_i}^G(d)- r_{1/4})^2 - (\sqrt2\tau_i - (kL - \bA_{\tau_i}(\xi)))^2 \\
    = \Big(2\sqrt2\tau_i - \tfrac{d+2}{2\sqrt2}\log \tau_i - r_{1/4} -(kL-\bA_{\tau_i}(\xi))\Big)\Big(\big(kL- \bA_{\tau_i}(\xi)\big) - \tfrac{d+2}{2\sqrt2}\log \tau_i - r_{1/4} \Big)\,.
\end{multline*}
Finish with $\tau_i \leq \tau_{TL}(\bA(\xi)) \leq KL^2$ and $\bA_{\tau_i}(\xi) \in [0, (KL^2)^{1/2+\nu}]$ on $E_{\xi}^{*}(K,L)$, defined in~\eqref{def:E*}.
\end{proof}
\end{claim}

Fix $\iota \in (0,2)$ and choose an interval $I\subset [L^{1.4}, \tau_{TL}(\bA(\xi))]$ of length $L^{2-\iota}$ such that the distance from $\sigma_k^*L^2$ to $I$ is at most $L^{2-\iota} \vee L^{1.4}$ (this is possible since $\sigma_k^*L^2 \leq  \tau_{TL}(\bA(\xi))$, as noted at the start of the proof).
We now show that, on $E_{\xi}^{**}(K,L) \cap E(k,\psi, I)$ and for any $\tau_i\in I$ and $v\in \cN_{\tau_i}^i$ as in~\eqref{def:percolation-event}, 
\begin{align}\label{eqn:good-angle-location}
    X_{\tau_i}(v) + A_{\tau_i}(\xi) \in (-kL,-kL+1/2] \quad \text{ and } \quad
    |\psi \cdot \arg(\Y_{\tau_i}(v)+ \Y_{\tau_i}(\xi))| \ \geq 1- \ep/2\,.
\end{align}
The first condition is immediate from the definition of $\mathbf{z}_{\tau_i}(k,\psi)$. Now,
$
|L^{-2}\tau_i - \sigma_k^*| \leq 2(L^{-\iota} \vee L^{-0.6})\,.
$
Using this, \eqref{eqn:bound-A-from-k}, and $\{\fC_L \leq K\}$, we find that for all $\tau_i \in I$ and $L$ sufficiently larger than $K$,
\begin{align*}
    kL-\bA_{\tau_i}(\xi) > kL - \bA_{\sigma_k^*L^2} - \big|L^{-1}\bA_{\sigma_k^*L^2} - L^{-1}\bA_{\tau_i}(\xi) \big| \geq \tfrac{\eta}{K}L- \fC_L| L^{-2} \tau_i - \sigma_k^*|^{\frac{1}2-\delta} \geq \tfrac{\eta}{2K}L\,.
\end{align*}
Since $\tau_i \geq L^{1.4}$, \cref{claim:perc-lb} yields
$
    \|\Y_{\tau_i}(v)\|^2 \geq \tfrac{\sqrt2 \eta}{K} L^{2.4} -CL^2\log L\,. 
$
Since $\|\Y_{\tau_i}(\xi)\|\leq L^{1/2+\nu}$ on $E_{\xi}^{\Bumpeq}(\nu,L)$ (defined in \eqref{def:spine-barrier-event}) and $\arg(\mathbf{z}_{\tau_i}(k,\psi)) = \psi$, simple trigonometry yields~\eqref{eqn:good-angle-location}.

Thus, on $E_{\xi}^{**}(K,L) \cap E(k,\psi, I)$, we have shown that, for $\tau_i\in I$ and $v\in\cN_{\tau_i}^i$ as in \eqref{def:percolation-event},
\[
    8^{-\frac{1}4}\hat{h}_{L,K}^{\ep}(k,\psi) \geq 8^{-\frac{1}4}L^{-\frac32}\|\Y_{\tau_i}(v)\| 
    \geq \sqrt{\frac{\tau_i}{L^2}\Big(k- \frac{\bA_{\tau_i}(\xi)}{L^2}\Big)- 8^{-\frac{1}{2}}CK^{\frac{1}2+\nu} L^{-1+4\nu}} > X_L(k) - \eta \,.
\]
In the last inequality, we have simply used H\"older continuity once more to replace $L^{-2}\tau_i$ with $\sigma_k^*$.
Finally, a union bound over all intervals of length $L^{2-\iota}$ contained in $[L^{1.4},\tau_{TL}(\bA(\xi))]$ yields 
\begin{align*}
    \Pb\Big( X_L(k) - 8^{-\frac{1}{4}}\hat{h}_{L,K}^{\ep}(k, \psi) > \eta \,,\, E_{\xi}^{**}(K,L) \Big) \leq KL^{\iota} \max_{I}\Pb\big(E_{\xi}^{**}(K,L) \cap E(k,\psi, I)) < e^{-CL^{2-\iota}}\,,
\end{align*}
for some constant $C>0$, where in the last inequality we used \cref{cor:gartner}. 
\qed

\subsection*{Acknowledgements and Declarations}
YHK acknowledges the support of the NSF Graduate Research Fellowship 1839302. OZ was supported by Israel Science Foundation grant 421/20 and a US-Israel BSF grant.
The authors have no conflicts of interest to declare that are relevant to the content of this article.

\bibliography{front.bib}
\bibliographystyle{abbrv}

\end{document}